\documentclass[aap,MSNbibl,citesort,dvips]{arximspdf}
\usepackage{algpseudocode}
\usepackage{algorithm,algorithmicx}
\usepackage{graphicx}

\doi{10.1214/11-AAP760}
\volume{21}
\issue{6}
\pubyear{2011}
\firstpage{2343}
\lastpage{2378}

\makeatletter

\algrenewcommand\alglinenumber[1]{\footnotesize \textbf{#1}}%

\def\bsuffix #1{#1}
\mathchardef\varTheta="0102

\newcommand{\eqref}[1]{(\ref{#1})}
\newcommand{\rmu}{\mathrm{u}}
\newcommand{\rmb}{\mathrm{b}}
\newcommand{\rmB}{\mathrm{B}}
\newcommand{\rmp}{\mathrm{p}}
\newtheorem{theorem}{Theorem}
\newtheorem{proposition}{Proposition}
\newtheorem{corollary}{Corollary}
\newtheorem{lemma}{Lemma}
\newcommand{\bbE}{\mathbb{E}}
\newcommand{\bbR}{\mathbb{R}}
\newcommand{\bbP}{\mathbb{P}}
\newcommand{\CF}{\mathcal{F}}
\newcommand{\rd}{\mathrm{d}}
\newcommand{\FS}{{\mathcal S}}

\newcommand{\Osc}{\operatorname{Osc}}
\newcommand{\conv}{\operatorname{conv}}
\newcommand{\Erf}{\operatorname{Erf}}
\makeatother

\begin{document}
\begin{frontmatter}

\title{Analysis of error propagation in particle filters with
approximation\thanksref{T1}}
\runtitle{Error analysis of particle filters with approximation}
\thankstext{T1}{Supported by the National Scientific
and Engineering Research Council of Canada (NSERC) through the
Discovery Grants program and the MITACS (Mathematics in Information
Technology and Complex Systems) Networked Centres of Excellence.}

\begin{aug}
\author[A]{\fnms{Boris N.} \snm{Oreshkin}\corref{}\ead[label=e1]{boris.oreshkin@mail.mcgill.ca}}
\and
\author[A]{\fnms{Mark J.} \snm{Coates}\ead[label=e2]{mark.coates@mcgill.ca}}
\runauthor{B. N. Oreshkin and M. J. Coates}
\affiliation{McGill University}
\address[A]{Department of Electrical\\ \quad and Computer
Engineering\\
McGill University\\
Montreal, Quebec H3A-2A7\\
Canada\\
\printead{e1}\\
\hphantom{\textsc{E-mail:}\ }\printead*{e2}}
\end{aug}

\received{\smonth{8} \syear{2009}}
\revised{\smonth{1} \syear{2011}}

%
\begin{abstract}
This paper examines the impact of approximation steps that become
necessary when particle filters are implemented on
resource-constrained platforms. We consider particle filters that
perform intermittent approximation, either by subsampling the
particles or by generating a parametric approximation. For such
algorithms, we derive time-uniform bounds on the weak-sense
$L_p$ error and present associated exponential
inequalities. We motivate the theoretical analysis by considering
the leader node particle filter and present numerical experiments
exploring its performance and the relationship to the error bounds.
\end{abstract}

%

\end{frontmatter}

\setcounter{footnote}{1}

\section{Introduction}
\label{sec:Introduction}

Particle filters have proven to be an effective approach for
addressing difficult tracking problems \cite{Dou01}. Since they are
more computationally demanding and require more memory than most
other filtering algorithms, they are really only a valid choice for
challenging problems, for which other well-established techniques
perform poorly. Such problems involve dynamics and/or observation
models that are substantially nonlinear and non-Gaussian. A
particle filter maintains a set of ``particles'' that are candidate
state values of the system (e.g., the position and velocity
of an object). The filter evaluates how well individual particles
correspond to the dynamic model and set of observations, and updates
weights accordingly. The set of weighted particles provides a
pointwise approximation to the filtering distribution, which
represents the posterior probability of the state.

The analysis of approximation error propagation and stability of
nonlinear Markov filters has been an active research area for
several decades \cite{Kun71,Oco96}. In the case of the particle
filter, there has been interest in establishing what conditions must
hold for the filter to remain stable (the error remaining bounded
over time), despite the error that is introduced at every time-step
of the algorithm by the pointwise approximation of the particle
representation \cite{Leg03,Gla04,Mor01,Mor02,Cri99,DelMor04,Douc05}.

In this paper, we focus on examining the impact of additional
intermittent approximation steps which become necessary when
particle filters are implemented on resource-constrained platforms.
The approximations we consider include subsampling of the particle
representation and the generation of parametric mixture models. The
main results of the paper are time-uniform bounds on the weak-sense
$L_p$-error induced by the combination of particle sampling error
and the additional intermittent approximation error (subsampling or
parametric). We employ the Feynman--Kac semigroup analysis
methodology described in \cite{DelMor04}; our investigation of
parametric approximation is founded on error bounds for the greedy
likelihood maximization algorithm, which was developed
in \cite{Li99} and analyzed in \cite{Zha03,Rak05}.

\subsection{Leader node particle filter}

Throughout the paper, we will motivate the analysis by considering
the concrete example of the ``leader node'' particle
filter \cite{Zhao03}, an algorithm that has been proposed for
distributed tracking in sensor networks. One of the major concerns
in distributed sensor network tracking is balancing the tradeoff
between tracking performance and network lifetime. The \textit{leader
node particle filter}, proposed in \cite{Zha02,Zhao03} and refined
and analyzed in \cite{Wil07,Ihl05}, achieves significant sensing and
communication energy savings. The leader node, which performs the
particle filtering, changes over time to follow the target and
activates only a subset of nodes at any time instant. Thus only the
active sensor nodes have to relay their measurements to a nearby
location.

%
\begin{figure}[t]

\includegraphics{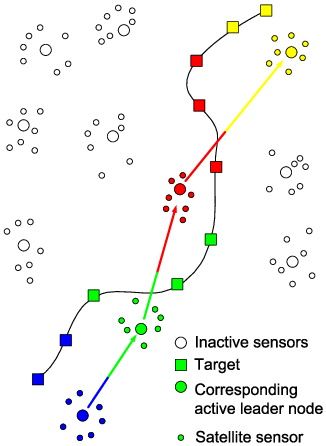}

\caption{The leader node distributed particle filtering setting.}
\label{fig:leader_node}
\end{figure}

The setting corresponding to this filtering paradigm is depicted in
Figure~\ref{fig:leader_node}. A leader node (depicted by the large
circles) is responsible for performing local tracking of the target
(trajectory depicted by squares) based on the data acquired by the
\textit{satellite} sensor nodes (depicted by small circles). The
satellite nodes have sensing capabilities and can locally transmit
the acquired data to the nearest leader node. The leader node fuses
the data gathered by the satellite nodes in its neighborhood,
incorporating them into its particle filter. Sensor management
strategies are used to determine when to change leader
node \cite{Wil07}. When this occurs, information must be exchanged
so that the new leader node can reconstruct the particle filter. In
attempting to alleviate the communication cost of transmitting all
particle values when the leader node is exchanged (which can involve
thousands of bits), the filtering distribution is more coarsely
approximated, either by transmitting only a subset of the particles
or by training a parametric model.

Mathematically, the leader node particle filter can be described as
follows. Suppose that $\mathcal{L} = \{1, 2, \ldots, L\}$ is the set
of possible leader nodes, and every leader node with label $\ell\in
\mathcal{L}$ has a set of satellite nodes $\mathcal{S}_{\ell}$ that
take measurements and transmit them to the leader node. The number
of such satellite nodes in the vicinity of the leader node $\ell$ is
$|\mathcal{S}_{\ell}|$. Denote by $\ell_t$ the label of the leader
node at time~$t$. We adopt the following state-space model to
describe the target evolution and measurement process:
%
%
\begin{eqnarray}\label{eqn:sig_proc_frame_1}
X_t &=& f_t(X_{t-1}, \varrho_t), \\
\label{eqn:sig_proc_frame_2}
Y_t^{j} &=& g_t^{j}(X_t, \zeta_t^{j}) \qquad\forall j \in
\mathcal{S}_{\ell_t} .
\end{eqnarray}
Here $X_t \in\bbR^{d_x}$ is the target state vector at time $t$,
$Y_t^j \in\bbR^{d_y^j}$ is the $j$th sensor measurement,
$\varrho_t$ and $\zeta_t^{j}$ are system excitation and measurement
noises, respectively, $f_t$ is a nonlinear system map $f_t \dvtx
{\mathbb{R}}^{d_x} \rightarrow{\mathbb{R}}^{d_x}$ and $g_t^j$ is a
nonlinear measurement map $g_t^j \dvtx {\mathbb{R}}^{d_x} \rightarrow
{\mathbb{R}}^{d_y^j}$. The target model is the same at every leader
node, but the observation process may be different.

\subsection{Feynman--Kac models}
Throughout the rest of this paper we adopt the methodology developed
in \cite{DelMor04} to analyze the behavior of filtering
distributions arising from (\ref{eqn:sig_proc_frame_1}) and
(\ref{eqn:sig_proc_frame_2}). This methodology involves representing
the particle filter as an $N$-particle approximation of a
Feynman--Kac model. We now briefly review the Feynman--Kac
representation (for a much more detailed description and discussion,
please refer to \cite{DelMor04}).

To describe the probabilistic model corresponding to the state-space
framework above, we need to introduce additional notation. Let
$(E_t,\mathcal{E}_t)$, $t\in\mathbb{N}$, be a sequence of measurable
spaces such that $X_t\in E_t$. Associated with a~measurable space of
the form $(E, \mathcal{E})$ is a set of probability measures~%
$\mathcal{P}(E)$ and the Banach space of bounded functions
$\mathcal{B}_\mathrm{b}(E)$ with supremum norm
\[
\| h \| = \sup_{x \in E} |h (x) |.
\]
We define a convex set $\Osc_1(E)$ of $\mathcal{E}$-measurable test
functions with finite oscillations
\begin{eqnarray*}
\operatorname{osc}(h) &=& \sup\bigl(| h(x) - h(y) |; x,y \in E\bigr), \\
\Osc_1(E) &=& \{ h \dvtx \operatorname{osc}(h) \leq1 \}.
\end{eqnarray*}
For any $h \in\mathcal{B}_\mathrm{b}(E)$ it is also possible to define the
following:
\[
\| h \|_{\operatorname{osc}} = \|h\| + \operatorname{osc}(h),
\]
and for a sequence of functions $(h_{i})_{1\leq i\leq N} \in
\mathcal{B}_\mathrm{b}(E)^N$ we define $\sigma^2(h)$ as
\[
\sigma^2(h) \triangleq\frac{1}{N} \sum_{k=1}^N
\operatorname{osc}^2(h_i).
\]
In order to simplify the representation, we define for a measure
$\mu\in\mathcal{P}(E)$,
\[
\mu(h) = \int_E h(x) \mu(dx)
\]
and for the Markov kernel from
$(E_{t-1}, \mathcal{E}_{t-1})$ to $(E_{t}, \mathcal{E}_t)$
\[
(\mu_{t-1} M_{t})(A_t) = \int\mu_{t-1}(\rd x_{t-1}) M_{i}(x_{t-1},
A_t).
\]

The target state vector in \eqref{eqn:sig_proc_frame_1} thus evolves
according to a nonhomogeneous discrete-time Markov chain $X_t$ with
transitions $M_{t}$ from $E_{t-1}$ into $E_{t}$. These transitions
and the initial distribution $\eta_0$ define the canonical
probability space
\[
\biggl( \Omega= \prod_{t\geq0} E_{t},
(\mathcal{F}_t)_{t\in\mathbb{N}}, (X_t)_{t\in\mathbb{N}},
\bbP_{\eta_0} \biggr),
\]
where the family of $\sigma$-algebras has the following property:
$\mathcal{F}_i \subset\mathcal{F}_j \subset\mathcal{F}_{\infty}$
for any $i \leq j$ and $\mathcal{F}_{\infty} = \sigma(\bigcup_{i\geq0}
\mathcal{F}_i)$. To characterize the properties of the observation
process in \eqref{eqn:sig_proc_frame_2} we introduce bounded and
nonnegative potential functions $G_t^{j} \dvtx E_t \rightarrow[0,
\infty), \forall j \in\FS_{\ell_t}$. Assuming that in every leader
node neighborhood $\FS_{\ell_t}$ observation noises, $\zeta_t^{j}$,
$j \in\FS_{\ell_{t}}$, in \eqref{eqn:sig_proc_frame_2} are
independent\footnote{The assumption of independence among the sensor
observations is not critical for the error analysis performed in the
paper, but is adopted because it allows for a more concrete
discussion and concise presentation of results.} the composite
potential function at leader node $\ell_t$ can be written via the
product of the individual potential functions of satellite nodes,
$G_{t}^{j}$
\[
G_{t}^{\FS_{\ell_t}} = \prod_{j \in\FS_{\ell_{t}}}
G_{t}^{j}.
\]
Then the propagation of the leader node Feynman--Kac model is
described by a pair of prediction-update operators. Since the
prediction operator is only concerned with the dynamics of the
target, it coincides with the Markov transition $M_{t}$ for all
possible leader nodes $\ell_t$ at time $t$. On the other hand, the
Boltzmann--Gibbs transformation on $\mathcal{P}(E_{t})$ is leader
node dependent and is defined for any $\eta\in\mathcal{P}(E_t)$
\[
\Psi_{t}^{\ell_{t}}(\eta)(\rd x_t) = \frac{
G_{t}^{\FS_{\ell_{t}}}(x_t) \eta(\rd x_t)}{\eta(
G_{t}^{\FS_{\ell_{t}}} )}.
\]

Using the diffusion $M_{t+1}$ and Boltzmann--Gibbs transformation
$\Psi_{t}^{\ell_{t}}$ we can identify an operator
$\Phi_{t+1}^{\ell_{t}} \dvtx \mathcal{P}(E_{t}) \rightarrow
\mathcal{P}(E_{t+1})$ which describes, for a given leader node
$\ell_t$, the evolution of the normalized prediction flow from time
$t$ to time $t+1$
\[
\Phi_{t+1}^{\ell_{t}}(\eta) = \Psi_{t}^{\ell_{t}}(\eta) M_{t+1}.
\]

To describe the evolution of the leader nodes we define a sensor
management rule $\Upsilon_{t}^{\ell_t} \dvtx \mathcal{P}(E_{t+1})
\times
\mathfrak{I}_t \rightarrow\mathcal{L} \times\{0,1\}$. This mapping
defines the next leader node, $\ell_{t+1} \in\mathcal{L}$, and the
decision, $\Delta_{t+1} \in\{0,1\}$, on whether or not the leader
node has to be exchanged: $\Delta_{t+1} = 1$ if we decide to
transfer the processing and measurement process to the leader node
other than the current one. Sensor management rules usually operate
according to the informativeness of the sensors and the predicted
position of the target.\footnote{We discuss an example of such an
algorithm, based on \cite{Wil07}, in Section~\ref{sec:experiments}.}
Thus the decision is made based on the utility of measurements
provided by different leader nodes given the current filtering
distribution $\Phi_{t+1}^{\ell_{t}}(\eta) \in\mathcal{P}(E_{t+1})$
and the information about leader nodes, $\mathcal{I}_t \in
\mathfrak{I}_t$. Information $\mathcal{I}_t$ may include coordinates
of nodes in the sensor network, measurement models for every node,
costs of performing a hand-off from the current leader node to other
leader nodes etc. The operation of the nonlinear mapping
$\Upsilon_{t}^{\ell_t}$ is described by the equation
\[
(\ell_{t+1}, \Delta_{t+1}) =
\Upsilon_{t}^{\ell_t}(\Phi_{t+1}^{\ell_{t}}(\eta), \mathcal{I}_t).
\]
The sequence of mappings $\Upsilon_{i-1}^{\ell_{i-1}}, \ldots,
\Upsilon_{t-2}^{\ell_{t-2}}$ defines the sequence\vspace*{-2pt} of leader nodes
$\ell_{i,t} = \ell_{i}, \ldots, \ell_{t-1}$ that can be used to
define the semigroups $\Phi_{i,t}^{\ell_{i,t}}$, $i\leq t$,
associated with the normalized Feynman--Kac distribution flows
\[
\Phi_{i,t}^{\ell_{i,t}} = \Phi_{t}^{\ell_{t-1}} \circ
\Phi_{t-1}^{\ell_{t-2}} \circ\cdots\circ\Phi_{i+1}^{\ell_{i}}.
\]
The semigroup $\Phi_{i,t}^{\ell_{i,t}}$ describes the evolution of
the normalized prediction Feyn\-man--Kac model from time $i$ to time
$t$ through the sequence of leader nodes $\ell_{i,t}$
\[
\eta_t^{\ell_{0,t}} = \Phi_{i,t}^{\ell_{i,t}}(\eta_i^{\ell_{0,i}}).
\]

Using the analysis tools developed in \cite{DelMor04}
$\Phi_{i,t}^{\ell_{i,t}}$ can further be related to potential
functions on $E_i$, $G_{i,t}^{\ell_{i,t}} \dvtx E_i \rightarrow(0,
\infty)$, and $P_{i,t}^{\ell_{i,t}} \dvtx \mathcal{P}(E_i) \rightarrow
\mathcal{P}(E_t)$,\vspace*{1pt} the Markov kernels from $E_i$ to $E_t$. In
particular, using the expectation with respect to the shifted chain,
\begin{eqnarray*}
&&\bbE_{i,x_i}\{ h_{i,t}(X_{i+1},\ldots,X_{t}) \}\\
&& \qquad  = \int
h_{i,t}(x_{i+1},\ldots,x_{t}) M_{i+1}(x_i, \rd x_{i+1}) \cdots
M_{t}(x_{t-1}, \rd x_{t}),
\end{eqnarray*}
and defining $G_{i,t}^{\ell_{i,t}}$ as
\[
G_{i,t}^{\ell_{i,t}}(x_i) = \bbE_{i, x_i} \biggl\{\prod
_{i\leq
j<t} G_j^{\FS_{\ell_j}}(X_j) \biggr\},
\]
we can introduce the multi-step Boltzmann--Gibbs transformation on
$E_i$ for any $\eta\in\mathcal{P}(E_i)$ and $h_i \in
\mathcal{B}_\mathrm{b}(E_i)$, $\Psi_{i,t}^{\ell_{i,t}}(\eta)(h_i) =
\eta(G_{i,t}^{\ell_{i,t}} h_i) / \eta(G_{i,t}^{\ell_{i,t}})$.
Defining~$P_{i,t}^{\ell_{i,t}}$ by the Feynman--Kac formulae,
\[
P_{i,t}^{\ell_{i,t}}(h_t) \propto\bbE_{i, x_i}\biggl \{ h_t(X_t)
\prod_{i\leq j<t} G_j^{\FS_{\ell_j}}(X_j) \biggr\},
\]
we can represent the semigroup $\Phi_{i,t}^{\ell_{i,t}}$ as follows:
\[
\Phi_{i,t}^{\ell_{i,t}}(\eta) = \Psi_{i,t}^{\ell_{i,t}}(\eta)
P_{i,t}^{\ell_{i,t}}.
\]

\subsection{Dobrushin contraction and regularity conditions}

The Dobrushin contraction coefficient ($\beta_{i,t}(P) \in[0,1]$)
plays a key role in our analysis. For a~fixed leader node sequence,
$\ell_{i,t}$, this can be defined as follows:
\[
\beta^{\ell_{i,t}}(P_{i,t}^{\ell_{i,t}}) = \sup\{\|
P_{i,t}^{\ell_{i,t}}(x_i,\cdot) - P_{i,t}^{\ell_{i,t}}(y_i,\cdot)
\|_{\mathrm{tv}}; x_i,y_i \in E_i \}.
\]
Here the total variation metric $\|\cdot\|_{\mathrm{tv}}$ is defined
for any $\mu,\eta\in\mathcal{P}(E)$ as $\|\mu(\cdot) -
\eta(\cdot)\|_{\mathrm{tv}} = \sup\{ |\mu(A) - \eta(A)| \dvtx A \in
\CF
\}$. We can also define the (worst case) Dobrushin contraction
coefficient, which is independent of the leader node sequence,
$\beta_{i,t}(P) = \sup_{\ell_{i,t}}
\beta^{\ell_{i,t}}(P_{i,t}^{\ell_{i,t}})$.

The estimation of the Dobrushin contraction coefficient is possible
if we adopt certain regularity assumptions regarding the components
of the Feynman--Kac operator. In particular, we adopt the following
condition on the Markov kernels:
\begin{itemize}
\item$(M)_{\rmu}^{(m)}$ There exists an integer $m \geq1$ and
strictly positive number $\epsilon_{\rmu}(M) \in(0,1)$ such
that for any $i \geq0$ and $x_i, y_i \in E_i$ we have
\[
M_{i,i+m}(x_i, \cdot) = M_{i+1}M_{i+2}\cdots M_{i+m}(x_i, \cdot) \geq
\epsilon_{\rmu}(M) M_{i,i+m}(y_i, \cdot).
\]
\end{itemize}
The following regularity condition is defined for the potential
functions:
\begin{itemize}
\item$(G)_{\rmu}$ There exists a strictly positive number
$\epsilon_{\rmu}(G) \in(0,1]$ such that for any $\ell_t$, $t
\geq0$ and $x_t, y_t \in E_t$
\[
G_t^{\mathcal{S}_{\ell_{t}}}(x_t) \geq\epsilon_{\rmu}^{K_{\rmu
}}(G)G_t^{\mathcal{S}_{\ell_{t}}}(y_t),
\]
\end{itemize}
$(G)_{\rmu}$ holds if a milder condition, $G_{t}^j(x_t) \geq
\epsilon_{t}(G_{t}^j)G_{t}^j(y_t)$, holds for all $t$, all potential
functions $G_{t}^{j}, j \in\mathcal{S}_{\ell_t}$, and for all
leader nodes $\ell_t \in\mathcal{L}$. In this case we can take
$\epsilon_{\rmu}(G) =
\inf_{t\geq0}\min_{\ell_{t} \in
\mathcal{L}}\min_{j \in\mathcal{S}_{\ell_t}}
\epsilon_{t}(G_{t}^j)$ and $K_{\rmu} = \max_{\ell_t
\in
\mathcal{L}} |\mathcal{S}_{\ell_t}|$.

The following two propositions that summarize results presented
in \cite{DelMor04}, Proposition~4.3.3, Corollary~4.3.3 and
Proposition~4.3.7, will be used for analyzing approximation error
propagation in the leader-node algorithm (being employed in the proofs
of Theorems~\ref{th:Lp_time_uni}, \ref{th:Exp_time_uni} and \ref
{th:Lp_param_time_uni}).
\begin{proposition}[(Proposition~4.3.3. and Corollary~4.3.3 \cite{DelMor04})]
When $(G)_{\rmu}$ and $(M)_{\rmu}^{(m)}$ are satisfied
we have for the Dobrushin contraction coefficient $\beta_{i,t}(P) =
\sup_{\ell_{i,t}} \beta^{\ell_{i,t}}(P_{i,t}^{\ell_{i,t}})$
%
%
\begin{equation} \label{eqn:dobrushin_onestep_bound}
\beta_{i,t}(P) \leq\bigl( 1 - \epsilon_{\rmu}^{2}(M)
\epsilon_{\rmu}^{(m-1)K_{\rmu}}(G) \bigr)^{\lfloor
(t-i)/m \rfloor}
\end{equation}
and the oscillations of the potential functions,
%
%
\begin{equation} \label{eqn:G_osc_bound}
 \quad \frac{\inf_{x_i \in E_i}
G_{i,t}^{\ell_{i,t}}(x_i)}{\|G_{i,t}^{\ell_{i,t}}\|} \geq
\epsilon_{\rmu}(M) \epsilon_{\rmu}^{m K_{\rmu}}(G)
, \qquad
\frac{\|G_{i,t}^{\ell_{i,t}}\|}{\nu(G_{i,t}^{\ell_{i,t}})}
\leq\epsilon_{\rmu}^{-1}(M) \epsilon_{\rmu}^{-m
K_{\rmu}}(G).
\end{equation}
\end{proposition}

\begin{proposition}[(Proposition~4.3.7 \cite{DelMor04})]
\label{prop:prop_437}
For any $0\leq p \leq n$, $\mu_p \in\mathcal{P}(E_p)$, and $f_n \in
\mathcal{B}_{\mathrm{b}}(E_n)$ with $\operatorname{osc}(f_n) \leq1$ there
exists a
function $f_{p,n}^{\mu_p}$ in $\mathcal{B}_{\mathrm{b}}(E_p)$ with
$\operatorname{osc}(f_{p,n}^{\mu_p}) \leq1$ such that for any $\eta
_p \in
\mathcal{P}(E_p)$ we have
%
%
\begin{equation}
| \Phi_{p,n}(\eta_{p}) - \Phi_{p,n}(\mu_{p}) | \leq\beta(P_{p,n})
\frac{\| G_{p,n} \|_{\operatorname{osc}}}{\eta_{p}(G_{p,n})} |
(\eta_{p} -
\mu_{p})(f_{p,n}^{\mu_p}) |.
\end{equation}
\end{proposition}

Proposition~\ref{prop:prop_437} implies that for any $\mu,\nu\in
\mathcal{P}(E_i)$ and $h_t \in\Osc_1(E_t)$ there exists $h_i \in
\Osc_1(E_i)$ such that
%
%
\begin{equation} \label{eqn:437_osc}
|[\Phi_{i,t}^{\ell_{i,t}}(\nu) - \Phi_{i,t}^{\ell_{i,t}}(\mu)](h_t)|
\leq\beta^{\ell_{i,t}}(P_{i,t}^{\ell_{i,t}})
\frac{\|G_{i,t}^{\ell_{i,t}}\|_{\operatorname{osc}}}{\nu
(G_{i,t}^{\ell_{i,t}})}
|(\nu- \mu)(h_i)|.
\end{equation}
Using the fact that we have for some positive function $\varphi$,
\[
\|\varphi\|_{\operatorname{osc}} = \|\varphi\| + \operatorname
{osc}(\varphi) \leq
\|\varphi\| \biggl( 2 - \frac{\inf_{y \in E}\varphi(y)}{\sup_{x \in
E}\varphi(x)} \biggr),
\]
and, furthermore, $\beta^{\ell_{i,t}}(P_{i,t}^{\ell_{i,t}}) \leq
\beta_{i,t}(P)$ for any $i \leq t$ and $\ell_{i,t}$, we see
%
%
\begin{eqnarray} \label{eqn:dobrushin_bound}
&&|[\Phi_{i,t}^{\ell_{i,t}}(\nu) -
\Phi_{i,t}^{\ell_{i,t}}(\mu)](h_t)|\nonumber
\\[-8pt]
\\[-8pt]
&& \qquad \leq\beta_{i,t}(P)
\frac{\|G_{i,t}^{\ell_{i,t}}\|}{\nu(G_{i,t}^{\ell_{i,t}})}\biggl [ 2
- \frac{\inf_{y_i \in
E_i}G_{i,t}^{\ell_{i,t}}(y_i)}{\|G_{i,t}^{\ell_{i,t}}\|} \biggr]
|(\nu- \mu)(h_{i})|.
\nonumber
\end{eqnarray}

Thus under assumptions $(G)_{\rmu}$ and
$(M)_{\rmu}^{(m)}$ the error propagation in the leader node
filter can be characterized as follows:
%
%
\begin{eqnarray} \label{eqn:prop_err_bound}
|[\Phi_{i,t}^{\ell_{i,t}}(\nu) - \Phi_{i,t}^{\ell_{i,t}}(\mu)](h_t)|
&\leq& \bigl( 1 - \epsilon_{\rmu}^{2}(M)
\epsilon_{\rmu}^{(m-1) K_{\rmu}}(G) \bigr)^{\lfloor
(t-i)/m \rfloor}\nonumber
\\[-8pt]
\\[-8pt]
&&{}\times\frac{2-\epsilon_{\rmu}(M)
\epsilon_{\rmu}^{m
K_{\rmu}}(G)}{\epsilon_{\rmu}(M)
\epsilon_{\rmu}^{m K_{\rmu}}(G)} |(\nu- \mu)(h_i)|.
\nonumber
\end{eqnarray}

These results describe the propagation of one-step approximation
error through the nonlinear operator $\Phi_{i,t}^{\ell_{i,t}}$.
They reveal the link between the initial error at time $i$ and the
propagated error at time $t$ through the properties of the potential
functions $G_{i,t}^{\ell_{i,t}}$ and the Dobrushin contraction
coefficient~$\beta_{i,t}(P)$.

\subsection{$N$-particle and parametric approximations}

Let the sampling operator $S^N \dvtx \mathcal{P}(E) \rightarrow
\mathcal{P}(E^N)$ be defined as
%
%
\begin{equation} \label{eqn:SN}
S^N(\eta)(h) = \frac{1}{N}\sum_{k=1}^N h(\xi_k) ,
\end{equation}
where $(\xi_1, \ldots, \xi_N)$ is the i.i.d. sample from $\eta$.
With this notation, the standard particle filter can be expressed
using the distribution update recursion, $\widehat\eta_{t+1} = S^N
(\Phi_{t+1}(\widehat\eta_{t}))$.

The operation of the leader node with additional approximations, on
the other hand, is more complex. In particular, the standard
particle filter recursion is applied if $\Delta_{t+1} = 0$ (leader
node does not change). If $\Delta_{t+1} = 1$, there is a change in
leader node, and there must be a transfer of information from the
current leader node to the next one. The communication of all $N$
particles is prohibitively costly in terms of energy. The leader
node particle filter therefore communicates a coarser approximation
of its $N$-particle representation. In this paper, we consider two
possibilities for this additional approximation step: (i) random
subsampling (choosing $N_{\rmb}$ of the particles at random);
and (ii) parametric approximation of the filtering
distribution.\vadjust{\goodbreak}

The subsampling leader node particle filter can then be expressed
as
%
%
\begin{eqnarray} \label{eqn:distr_PF_NP}
\widehat\eta_{t+1}^{\ell_{0,t+1}^{\prime}} &=& S^N \circ
S^{N_{\rmb}}(\Phi_{t+1}^{\ell_t^{\prime}}(\widehat\eta
_{t}^{\ell_{0,t}^{\prime}}))  \qquad \mbox{if } \Delta_{t+1}^{\prime
} = 1, \nonumber
\\[-8pt]
\\[-8pt]
\widehat\eta_{t+1}^{\ell_{0,t+1}^{\prime}} &=&
S^N(\Phi_{t+1}^{\ell_t^{\prime}}(\widehat\eta_{t}^{\ell
_{0,t}^{\prime}}))
 \qquad \mbox{if } \Delta_{t+1}^{\prime} = 0.
\nonumber
\end{eqnarray}
Here $\widehat\eta_{t}^{\ell_{0,t}^{\prime}}$ is the distribution
obtained via the sequence of the leader\vspace*{-2pt} nodes~$\ell_{0,t}^{\prime}$
with the convention
$\Phi_{0}^{\ell_{-1}^{\prime}}(\widehat\eta_{-1}^{\ell
_{0,-1}^{\prime}})
= \eta_{0}$ and
$\Phi_{1}^{\ell_{0}^{\prime}}(\widehat\eta_{0}^{\ell
_{0,0}^{\prime}})
= S^N(\eta_{0}) M_{1}$. In this scenario, the sensor management step
is accomplished via a suboptimal rule using the approximate
prediction of the target state
\[
(\ell_{t+1}^{\prime}, \Delta_{t+1}^{\prime}) =
\Upsilon_{t}^{\ell_t^{\prime}}(\Phi_{t+1}^{\ell_t^{\prime
}}(\widehat\eta_{t}^{\ell_{0,t}^{\prime}}),
\mathcal{I}_t).
\]
There is also an additional subsampling operation
($S^{N_{\rmb}}$) after the update of the predictive posterior
using the operator $\Phi_{t+1}^{\ell_t^\prime}$. Note that
$N_{\rmb} < N$ so that the communication cost of the leader
node exchange is reduced, since only $N_{\rmb}$ particles are
transmitted. This step is followed by communication of the
subsampled particle set to the new leader node, and finally there is
an upsampling operation to regenerate $N$ particles from the
$N_\mathrm{b}$-particle approximation.

In order to express the parametric approximation particle filter in
an analogous fashion, we introduce an operator
$\mathbb{W}_{N_{\rmp}} \dvtx \mathcal{P}(E) \rightarrow
\mathcal{P}(E^{N_{\rmp}})$, which, when applied to a measure
$\nu\in\mathcal{P}(E)$, constructs a parametric mixture
approximation comprised of $N_{\rmp}$ mixture components
%
%
\begin{equation} \label{eqn:WN}
\mathbb{W}_{N_{\rmp}}(\nu)(h) =
\sum_{k=1}^{N_{\rmp}} \alpha_k \mu_{\theta_k}(h) .
\end{equation}
Here $\mu_{\theta_k} \in\mathcal{P}(E)$ is parameterized by a set
of parameters $\theta_k$ and $\alpha_k$ are weights satisfying
$\alpha_k \geq0$ and $\sum_k \alpha_k = 1$; $\theta_k$ and
$\alpha_k$ are estimated from $\nu$. Section~\ref{sec:parametric}
provides a concrete example of $\mathbb{W}_{N_{\rmp}}$ based
on the greedy maximum likelihood maximization. The parametric
approximation particle filter can then be expressed as
%
%
\begin{eqnarray} \label{eqn:distr_PF_PA}
\widehat\eta_{t+1}^{\ell_{0,t+1}^{\prime}} &=& S^N \circ\mathbb
{W}_{N_{\rmp}} \circ S^N (\Phi_{t+1}^{\ell_t^{\prime
}}(\widehat\eta_{t}^{\ell_{0,t}^{\prime}}))  \qquad \mbox{if }
\Delta_{t+1}^{\prime} = 1, \nonumber
\\[-8pt]
\\[-8pt]
\widehat\eta_{t+1}^{\ell_{0,t+1}^{\prime}} &=&
S^N(\Phi_{t+1}^{\ell_t^{\prime}}(\widehat\eta_{t}^{\ell
_{0,t}^{\prime}}))
 \qquad \mbox{if } \Delta_{t+1}^{\prime} = 0.
\nonumber
\end{eqnarray}
Here if there is a leader node exchange ($\Delta_{t+1}^{\prime} =
1$) the output of the standard particle filter is fed into the
parametric mixture approximation operator that outputs parameters
$\theta_k$ and weights $\alpha_k$, $k = 1,\ldots, N_{\rmp}$. These
parameters and weights are further transmitted to the new leader
node to reduce the communication cost. An $N$-particle approximation
is then regenerated by sampling from the mixture with parameters
$\theta_k$ and weights $\alpha_k$.

\subsection{Problem statement}

In this paper we study the additional approximation errors arising
during the leader node exchanges. These additional approximation
errors are the result of either additional random subsampling in the
subsample approximation leader node particle filter, or the
additional parametric approximation in the parametric leader node
particle filter.

Let us denote $\eta_t^{\ell_{0,t}} =
\Phi_{0,t}^{\ell_{0,t}}(\eta_{0})$ the true leader node distribution
flow and~$\ell_{0,t}$ the associated sequence of leader nodes
obtained via the optimal sensor management rule $(\ell_{t+1},
\Delta_{t+1}) =
\Upsilon_{t}^{\ell_t}(\Phi_{t+1}^{\ell_{t}}(\eta_t^{\ell_{0,t}}),
\mathcal{I}_t)$. The approximate\vspace*{-2pt} leader node distribution flow
$\widehat\eta_t^{\ell_{0,t}^{\prime}}$ defined by either
\eqref{eqn:distr_PF_NP} or \eqref{eqn:distr_PF_PA} uses the sequence
of leader nodes obtained via the suboptimal sensor management rule,\vspace*{-2pt}
$(\ell_{t+1}^{\prime},\allowbreak \Delta_{t+1}^{\prime}) =
\Upsilon_{t}^{\ell_t^{\prime}}(\Phi_{t+1}^{\ell_{t}^{\prime
}}(\widehat\eta_t^{\ell_{0,t}^{\prime}}),
\mathcal{I}_t)$.

The global error between the true filtering distribution,
$\eta_t^{\ell_{0,t}}$, and $\widehat\eta_t^{\ell_{0,t}^\prime}$ can
be split into two components:
\begin{eqnarray*}
\bbE\{|[\eta_t^{\ell_{0,t}} -
\widehat\eta_t^{\ell_{0,t}^\prime}](h_t)|^p\}^{1/p} &\leq&
\bbE\{|[\eta_t^{\ell_{0,t}^\prime} -
\widehat\eta_t^{\ell_{0,t}^\prime}](h_t)|^p\}^{1/p}\\
&&{} +
\bbE\{|[\eta_t^{\ell_{0,t}} -
\eta_t^{\ell_{0,t}^\prime}](h_t)|^p\}^{1/p}.
\end{eqnarray*}
Here the first term represents the error accumulated in the leader
node recursion because of the additional distribution approximations
during the leader node exchanges, and the second term represents the
errors arising due to the sub-optimality of the sensor management
rule. In this paper we study the errors of the first kind.

The global error of the first kind,
$\widehat\eta_t^{\ell_{0,t}^\prime} - \eta_t^{\ell_{0,t}^\prime}$,
can be related to the sequence of local approximation errors
$\widehat\eta_i^{\ell_{0,i}^\prime} -
\Phi_{i}^{\ell_{i-1}^\prime}(\widehat\eta_{i-1}^{\ell
_{0,i-1}^\prime}),
i = 0,\ldots,t$ \cite{DelMor04}
%
%
\begin{equation} \label{eqn:err_decomp}
\widehat\eta_t^{\ell_{0,t}^\prime} - \eta_t^{\ell_{0,t}^\prime} =
\sum_{i=0}^t [
\Phi_{i,t}^{\ell_{i,t}^\prime}(\widehat\eta_i^{\ell_{0,i}^\prime
}) -
\Phi_{i,t}^{\ell_{i,t}^\prime}(\Phi_{i}^{\ell_{i-1}^\prime
}(\widehat\eta_{i-1}^{\ell_{0,i-1}^\prime}))
].
\end{equation}
To simplify the notation in the rest of the article, we will use the
following convention, suppressing the explicit identification of the
leader-node\vspace*{-1pt} sequences; we will write $\eta_t^{\prime} \equiv
\eta_t^{\ell_{0,t}^\prime}$ and $\eta_t \equiv
\eta_t^{\ell_{0,t}}$, with associated mappings\vspace*{-1pt} $\Phi_t \equiv\Phi
_{t}^{\ell_{t-1}}$
and $\Phi_{t}^{\prime} \equiv\Phi_{t}^{\ell_{t-1}^\prime}$. Similarly
we express the particle approximations as $\widehat\eta_t^{\prime}
\equiv
\widehat\eta_t^{\ell_{0,t}^\prime}$ and $\widehat\eta_t \equiv
\widehat\eta_t^{\ell_{0,t}}$.

\subsection{Paper organization}

The rest of the paper is organized as follows.
Section~\ref{sec:foundational} presents some foundational results
that serve as the basis for our analysis. In
Section~\ref{sec:bootstrap} we present error bounds and exponential
inequalities for the leader node particle filter that performs
intermittent subsampling, and in Section~\ref{sec:parametric} we
analyze the performance of this filter when it employs parametric
approximation. Section~\ref{sec:experiments} describes numerical
experiments that illustrate the performance of the algorithms we
analyze and the relationship to the bounds.
Section~\ref{sec:related} discusses related work, and
Section~\ref{section::Concluding Remarks} summarizes the
contribution and makes concluding remarks.

\section{Bounds on errors induced by sampling}
\label{sec:foundational}

The following result bounds the weak-sense $L_p$ error induced by
the sampling operator for functions with finite oscillations. It is
used to characterize the one-step approximation errors in the leader
node particle filter.

\begin{lemma} \label{lem:2}
Suppose $\nu\in\mathcal{P}(E)$, then for any $p \geq1$ and an
$\mathcal{E}$-measurable function $h$ with finite oscillations we
have
\[
\mathbb{E}\{ |[\nu- S^N(\nu)](h)|^p \}^{ {1}/{p}} \leq
c(p)^{ {1}/{p}} \frac{\sigma(h)}{\sqrt{N}},
\]
where $c(p)$ is defined as follows:
\[
c(p) =
\cases{\displaystyle 1 ,&\quad if $p = 1 $,\cr\displaystyle
2^{-p/2} p \Gamma[p/2] ,&\quad if $p > 1$,
}
\]
and $\Gamma[\cdot]$ is the Gamma function.
\end{lemma}

\begin{pf}
Since $\mathbb{E}\{[\nu- S^N(\nu)](h)\} = \nu(h) - \nu(h) = 0$, we
have, from the Chernov--Hoeffding inequality,
\[
\mathbb{P}\{ |[\nu- S^N(\nu)](h)| \geq\epsilon\} \leq2
e^{- {2 N \epsilon^2}/({\sigma^2(h)})}.
\]
We note that
\[
\mathbb{P}\{ |[\nu- S^N(\nu)](h)|^p \geq\epsilon\} = \mathbb{P}\{
|[\nu- S^N(\nu)](h)| \geq\epsilon^{1/p} \},
\]
and we have, from the Chernov--Hoeffding inequality,
\[
\mathbb{P}\{ |[\nu- S^N(\nu)](h)| \geq\epsilon^{1/p} \} \leq2
e^{-2 N \epsilon^{2/p}/\sigma^2(h)}.
\]
Next we recall the following property:
\[
\mathbb{E}\{ |[\nu- S^N(\nu)](h)| \} = \int_0^\infty\mathbb{P}\{
|[\nu- S^N(\nu)](h)| \geq\epsilon\} \,\rd\epsilon.
\]
And finally we obtain
\begin{eqnarray*}
\mathbb{E}\{ |[\nu- S^N(\nu)](h)|^p \}^{ {1}/{p}} &=& \biggl[
\int_0^\infty\mathbb{P}\{ |[\nu- S^N(\nu)](h)| \geq\epsilon^{1/p}
\} \,\rd\epsilon
\biggr]^{ {1}/{p}} \\
&\leq& \biggl[ 2 \int_0^\infty e^{-2
N \epsilon^{2/p}/\sigma^2(h)} \,\rd\epsilon
\biggr]^{ {1}/{p}} \\
&=& \biggl[ \sigma^p(h) p (2N)^{- {p}/{2}}
\Gamma\biggl[\frac{p}{2} \biggr] \biggr]^{ {1}/{p}}.
\end{eqnarray*}
Applying Lemma 7.3.3 of \cite{DelMor04} allows us to set $c(1) = 1$
instead of $c(1) = 2^{-1/2} \Gamma[1/2] = \sqrt{\pi/2}$, and this
completes the
proof.
\end{pf}

Lemma~\ref{lem:2} tightens Lemma 7.3.3 from \cite{DelMor04} and
extends it to include noninteger~$p$. It is relatively
straightforward to see why the sequence of constants~$c(p)$ provides
improvement over Lemma 7.3.3 from \cite{DelMor04} that uses the
sequence of constants $d(p)$. For example, for even $p=2n$,
$d(2n) = (2n)!/n! 2^{-n}$ and the ratio of the two sequences is
%
%
\begin{equation} \label{eqn:cp_dp}
 \qquad \frac{d(2n)}{c(2n)} = \frac{(2n)!2^{-n}}{n!(2n)\Gamma(n)2^{-n}} =
\frac{(2n-1)!}{n(n-1)!\Gamma(n)} =
\frac{\Gamma(2n)}{n\Gamma(n)\Gamma(n)} = \frac{1}{n \rmB(n,n)}.
\end{equation}
Here $\rmB$ is the Beta function. $\rmB(n,n)$ is a quickly
decaying function. For large~$n$, Stirling's approximation gives a
simple expression for the Beta function, $\rmB(n,n) \sim
\sqrt{2\pi}n^{-1/2}2^{-2n+1/2}$, yielding the large $n$ Stirling's
approximation for (\ref{eqn:cp_dp}),
\[
\frac{d(2n)}{c(2n)} \sim\frac{1}{\sqrt{2\pi n}} 2^{2n-1/2}.
\]
This shows that $c(p)$ grows much slower with $p$ than $d(p)$.

The following theorem provides a bound on the moment-generating
function of the empirical process $\sqrt{N}[\nu- S^N(\nu)](h)$. The
result employs Lemma~\ref{lem:2} to tighten Theorem~7.3.1
of \cite{DelMor04}.
\begin{theorem} \label{th:3}
For any $\mathcal{E}$-measurable function $h$ such that $\sigma(h) <
\infty$, we have for any $\varepsilon$
\[
\mathbb{E} \bigl\{ e^{\varepsilon\sqrt{N}|[\nu- S^N(\nu)](h)|}
\bigr\} \leq1 + \varepsilon\sigma(h) \Biggl(1-\sqrt{\frac{\pi}{2}}
+ \sqrt{\frac{\pi}{2}}e^{ ({\varepsilon^2}/{8})\sigma^2(h)}
\biggl[ 1
+ \Erf\biggl[ \frac{\varepsilon\sigma(h)}{\sqrt{8}} \biggr]
\biggr]
\Biggr).
\]
\end{theorem}

\begin{pf}
We first utilize the power series representation of the exponential
\[
\mathbb{E} \bigl\{ e^{\varepsilon|[\nu- S^N(\nu)](h)|} \bigr\} =
1 + \varepsilon\mathbb{E} \{|[\nu- S^N(\nu)](h)| \} +
\sum_{n \geq2} \frac{\varepsilon^{n}}{n!}
\mathbb{E} \{|[\nu- S^N(\nu)](h)|^{n} \}.
\]
Utilizing Lemma~\ref{lem:2} we have
\begin{eqnarray*}
&&\mathbb{E} \bigl\{ e^{\varepsilon|[\nu- S^N(\nu)](h)|} \bigr\}
\\
&& \qquad\leq1 + \frac{\varepsilon\sigma(h)}{\sqrt{N}} + \sum_{n
\geq2} \biggl[\frac{\varepsilon\sigma(h)}{(2N)^{1/2}} \biggr]^{n}
\frac{\Gamma[ n/2 ]}{(n-1)!} \\
&& \qquad= 1 + \frac{\varepsilon\sigma(h)}{\sqrt{N}} -
\frac{\varepsilon\sigma(h) \sqrt{\pi}}{\sqrt{2N}} +
\frac{\varepsilon\sigma(h)
\sqrt{\pi}}{\sqrt{2N}}e^{ {\varepsilon^2\sigma
^2(h)}/({8N})} \biggl[
1 + \Erf\biggl[ \frac{\varepsilon\sigma(h)}{\sqrt{8N}} \biggr]
\biggr].
\end{eqnarray*}
Choosing $\varepsilon= \varepsilon\sqrt{N}$ and rearranging terms
completes the proof.\vadjust{\goodbreak}
\end{pf}

The following corollary, containing a more tractable variation of the
previous theorem, can be useful for deriving the exponential
inequalities for the particle approximations of Feynman--Kac models.
\begin{corollary} \label{cor:simple_MGF_bound}
For any $\mathcal{E}$-measurable function $h$ such that $\sigma(h)
< \infty$, we have for any $\varepsilon$
\[
\mathbb{E} \bigl\{ e^{\varepsilon\sqrt{N}|[\nu- S^N(\nu)](h)|}
\bigr\} \leq\bigl(1 + \sqrt{2\pi} \varepsilon\sigma(h) \bigr)
e^{ ({\varepsilon^2}/{8})\sigma^2(h)}.
\]
\end{corollary}

\begin{pf}
The proof is straightforward since $\sup_{x}\Erf(x) = 1$,
$1-\sqrt{\pi/2} < 0$ and $e^{ ({\varepsilon^2}/{8})\sigma^2(h)}
\geq1$.
\end{pf}

We note that the simplified estimate of the moment-generating function in
Corollary~\ref{cor:simple_MGF_bound} is much tighter than the bound in
Theorem~7.3.1 of \cite{DelMor04} for asymptotically large deviations
$\varepsilon$ while the more complex bound in Theorem~\ref{th:3}
is uniformly tighter over the range of $\varepsilon$.

\section{Particle filters with intermittent subsampling}
\label{sec:bootstrap}

This section presents an analysis of the error propagation in the
leader node particle filter that performs intermittent subsampling
approximation steps. We focus on the case where the number of
particles $N$ is constant, and the subsampling approximation step
always uses $N_\mathrm{b}$ particles. Our main results are a time-uniform
bound on the weak-sense $L_p$-error and an associated exponential
inequality.

\subsection{Time-uniform error bounds and exponential inequalities}
\label{ssec:bootstrap_time_uni}

We now analyze the global approximation error for the leader node
particle filter with intermittent subsampling defined by
recursion~\eqref{eqn:distr_PF_NP}. We first present a~theorem that
specifies a time-uniform bound on the weak-sense $L_p$ error.
\begin{theorem} \label{th:Lp_time_uni}
Suppose $\widehat\eta_t^{\prime}$ is defined
by~\eqref{eqn:distr_PF_NP} and assumptions $(G)_{\rmu}$\break and~%
$(M)_{\rmu}^{(m)}$ hold. Suppose further that $\mathbb{P}\{
\Delta^{\prime}_{i} = 1\} \leq q_{\rmu}$ for any $i \geq0$
and $0\leq q_{\rmu} \leq2/3$. Then for a positive integer
$\chi$ such that $N = \chi N_{\rmb}$, $t \geq0$, $p \geq1$
and $h_t \in\Osc_1(E_t)$ we have the time-uniform estimate
\[
\sup_{t \geq0}\mathbb{E} \{ |[\widehat\eta_t^{\prime
} -
\eta_t^{\prime}](h_t)|^p \}^{1/p} \leq
\frac{\epsilon_{\rmu,m} c^{1/p}(p) }{\sqrt{N}} \bigl(
q_{\rmu}^{1/p}\sqrt{\chi} + (1-q_{\rmu})^{1/p} \bigr),
\]
where the constant $\epsilon_{\rmu,m}$ is
%
%
\begin{equation}\label{eq:epsdef}
\epsilon_{\rmu,m} = m\bigl(2 - \epsilon_{\rmu}(M) \epsilon^{m
K_{\rmu}}_{\rmu}(G)\bigr) / \epsilon^3_{\rmu}(M)
\epsilon^{(2m-1)K_{\rmu}}_{\rmu}(G).
\end{equation}
\end{theorem}

\begin{pf}
This and other technical proofs can be found in
Section~\ref{sec:proofs}.
\end{pf}

The result can be generalized to cases where $N$ is not an integer
multiple of $N_\mathrm{b}$, at the expense of a slight loosening of the bound.

\begin{corollary} \label{cor:Lp_time_uni} Suppose the assumptions
of Theorem~\ref{th:Lp_time_uni} apply, except we allow any integer
${N_{\rmb}} < N$. Then for any $t \geq0$, $p \geq1$
and $h_t \in\Osc_1(E_t)$ we have the time-uniform estimate
\begin{eqnarray*}
&&\sup_{t \geq0}\mathbb{E} \{ |[\widehat\eta_t^{\prime
} -
\eta_t^{\prime}](h_t)|^p \}^{1/p}\\
&& \qquad  \leq\epsilon_{\rmu,m}
c^{1/p}(p) \biggl( q_{\rmu}^{1/p} \biggl[\frac{1}{\sqrt{N}} +
\frac{1}{\sqrt{{N_{\rmb}}}} \biggr] + (1-q_{\rmu})^{1/p}
\frac{1}{\sqrt{N}} \biggr).
\end{eqnarray*}
\end{corollary}

\begin{corollary} \label{cor:Lp_time_uni_tighter} Under the same
assumptions as Theorem~\ref{th:Lp_time_uni}, we have for $p \in
\mathbb{N}$ and
$h_t \in\Osc_1(E_t)$ the time-uniform estimate
%
%
\begin{equation} \label{eqn:corLp}
\sup_{t \geq0}\mathbb{E} \{ |[\widehat\eta_t^{\prime
} -
\eta_t^{\prime}](h_t)|^p \}^{1/p} \leq
\frac{\epsilon_{\rmu,m} c^{1/p}(p) }{\sqrt{N}} \bigl(
q_{\rmu} \chi^{p/2} + (1-q_{\rmu}) \bigr)^{1/p}.
\end{equation}
\end{corollary}

The intuitive implication of Theorem~\ref{th:Lp_time_uni} and
Corollary~\ref{cor:Lp_time_uni_tighter} is that rare approximation
events have limited effect on the average approximation error of the
subsample leader node particle filter. The $L_2$ error bound for the
standard particle filter is the same as (\ref{eqn:corLp}) of
Corollary~\ref{cor:Lp_time_uni_tighter} taken with $p=2$, except for
the term $ (q_{\rmu} \chi+ (1-q_{\rmu})
)^{1/2}$. This expression thus quantifies the performance
deterioration, in terms of $L_2$ error bounds, due to the subsample
approximation step.

If the compression factor, $\chi$, is $\chi= 10$, and subsample
approximations occur with probability $0.1$, then the deterioration
of the approximation error captured, in terms of bounds, by the
factor $ (0.1 \times10 + (1-0.1) )^{1/2}$, is around
$40\%$. The communication overhead, on the other hand, represented
by the total number of particles transmitted during leader node
hand-off, is reduced by a factor of $10$. The compressed particle
cloud exchanges are most efficient in scenarios where the targets
being tracked have slow dynamics and the density of leader nodes is
relatively low (both implying rare hand-off events), but the
tracking accuracy requirements and leader-to-leader communication
costs are high.

Theorem~\ref{th:Exp_time_uni} below provides the exponential
estimate for the probability of large deviations of the approximate
Feynman--Kac flows associated with the subsample approximation
particle filter.

\begin{theorem} \label{th:Exp_time_uni}
Suppose assumptions $(G)_{\rmu}$ and $(M)_{\rmu}^{(m)}$
hold. Suppose further that $\mathbb{P}\{ \Delta^{\prime}_{i} = 1\}
\leq q_{\rmu}$ for $i \geq0$ and $0\leq q_{\rmu} \leq
1$. Then for any ${N_{\rmb}} < N$, $t \geq0$ and $h_t \in
\Osc_1(E_t)$ we have
\begin{eqnarray*}
\sup_{t \geq0} \mathbb{P} \{ |[\widehat\eta
_t^{\prime}
- \eta_t^{\prime}](h_t)| \geq\epsilon\} &\leq& \biggl(1 +
4\sqrt{2\pi}\frac{\varepsilon
\sqrt{N}}{\epsilon_{\rmu,m}} \biggr) e^{- {N
\varepsilon^2}/({2\epsilon_{\rmu,m}^2})} \\
&&{}+ q_{\rmu} \biggl(1 + 4\sqrt{2\pi}\frac{\varepsilon
\sqrt{N_{\rmb}}}{\epsilon_{\rmu,m}} \biggr)
e^{- {N_{\rmb} \varepsilon^2}/({2\epsilon_{\rmu,m}^2})}.
\end{eqnarray*}
\end{theorem}

The implication of this theorem is that the tail probabilities of the
approximation error can be significantly affected by the rare hand-off
events. Although the average approximation error bounds obtained in
Corollary~\ref{cor:Lp_time_uni_tighter} appear encouraging, care
should be exercised when selecting approximation parameters to prevent
the explosion of the tails of the approximation error
distribution. These tails characterize the probabilities of relatively
rare, but catastrophic events.

\section{Particle filtering with intermittent parametric approximations}
\label{sec:parametric}

In this section we analyze the error behavior of the leader node
particle filter described by recursion \eqref{eqn:distr_PF_PA}. This
filter incorporates intermittent parametric mixture estimation of
the filtering probability density. The probability density
estimation problem consists of estimating an unknown probability
density given the i.i.d. sample $\{\xi_i\}_{1\leq i\leq N}$ from
this density. As before, let $(E, \mathcal{E})$ be a measurable
space. Denote $\lambda$ a $\sigma$-finite measure on $\mathcal{E}$.
Throughout this section it is assumed that the underlying
distribution has a density if its Radon--Nikodym derivative with
respect to $\lambda$ exists.

We assume that with the sequence of the approximate filtering
distributions, $\Phi_{i+1}^{\prime}(\widehat\eta_{i}^{\prime})$,
there exists an associated and well-behaved sequence of approximate
filtering densities $\frac{\rd}{\rd
x_{i+1}}\Phi_{i+1}^{\prime}(\widehat\eta_{i}^{\prime})$ so that the
mixture density estimation problem is well defined. The main result
of the section, constituted in Theorem~\ref{th:Lp_param_time_uni},
is a time-uniform, weak-sense $L_p$ error bound characterizing the
expected behavior of the parametric approximation leader node
particle filter.

\subsection{Parametric approximation}

Within the Greedy Maximum Likelihood (GML) framework proposed by Li
and Barron \cite{Li99}, the discrepancy between the target density
$f$ and its estimate is measured by the Kullback--Leibler (KL)
divergence. For any two measures $\nu$ and $\mu$ on $E$,
KL-divergence can be defined as follows:
%
%
\begin{equation}
D( \nu\| \mu) = \int\log\frac{\mathrm{d}\nu}{\mathrm{d}\mu}
\,\mathrm{d}\nu.
\end{equation}
We will also abuse notation by writing KL-divergence for two
arbitrary densities $f$ and $g$ in a similar fashion
%
%
\begin{equation}
D( f \| g ) = \int\log\frac{f(x)}{g(x)} f(x) \,\rd x.
\end{equation}
Consider the following class of bounded parametric probability
densities:
\[
\mathcal{H}_i = \Bigl\{ \phi_{\theta_i}(x) \dvtx \theta_i \in\Theta_i,
a_i \leq\inf_{\theta_i, x_i}\phi_{\theta_i}(x_i),
\sup_{\theta_i, x_i}\phi_{\theta_i}(x_i) \leq b_i \Bigr\},
\]
where $0 < a_i < b_i < \infty$ and $\Theta_i \subset
\mathbb{R}^{d_i}$ defines the parameter space, and $\inf$ and $\sup$
are taken over $\Theta_i$ and $E_i$. In the setting where the\vadjust{\goodbreak}
intermittent approximation is accomplished using parametric
approximation, we are looking for a sequence of mixture density
estimators of the filtering densities. We thus define the class of
bounded parametric densities, $\phi_{\theta_i}(x)$, indexing it by
time-step $i$ to emphasize that the parameterization can be
time-varying. The approximation is restricted to a class of discrete
$N_{\rmp}$-component convex combinations of the form
\begin{eqnarray*}
\mathcal{C}_{N_{\rmp},i} &=&
\conv_{N_{\rmp}}(\mathcal{H}_i)\\
 &=& \Biggl\{ g \dvtx g(x) =
\sum_{j=1}^{N_{\rmp}} \alpha_{i,j}
\phi_{\theta_{i,j}}(x), \phi_{\theta_{i,j}} \in\mathcal{H}_i,
\sum_{j=1}^{N_{\rmp}} \alpha_{i,j} = 1,
\alpha_{i,j} \geq0 \Biggr\}.
\end{eqnarray*}
As $N_{\rmp}$ grows without bound,
$\mathcal{C}_{N_{\rmp},i}$ converges to the class of
continuous convex combinations
\[
\mathcal{C}_i = \conv(\mathcal{H}_i) = \biggl\{g \dvtx g(x) =
\int_{\Theta} \phi_{\theta_i}(x) \mathbb{P}(\mathrm{d}\theta_i),
\phi_{\theta_i} \in\mathcal{H}_i \biggr\}.
\]

The general framework for the greedy approximation of arbitrary cost
functions is discussed in \cite{Zha03}. The particular instance of
this more general framework is the GML for mixture approximation
(see \cite{Li99}). The corresponding computational routine, a
sequential greedy maximum likelihood, associated with this procedure
and based on the sample $(\xi_i)_{1\leq i\leq N}$ from the target
density $f$ is summarized in the form of
Algorithm~\ref{alg:greed_ML}. The optimization step in this
algorithm can be performed with any standard numerical nonlinear
optimization technique.

%
%
%
%

\begin{algorithm}[t]
\caption{GML}\label{alg:greed_ML}
\begin{algorithmic}[1]
\State Given $g_{1} \in\mathcal{H}$,
\For {$k=2$ to $N_{\rmp}$}
\State Find $\phi_{\theta_k} \in\mathcal{H}$ and $0 \leq\alpha_k \leq1$ to maximize the function
\State $(\theta_k^*, \alpha_k^*) = \arg\max_{\alpha_k, \theta_k} \sum_{j=1}^N \log((1-\alpha_k)g_{k-1}(\xi_j) + \alpha_k\phi_{\theta_k}(\xi_j))$.
\State Let $g_k = (1-\alpha_k^*)g_{k-1} + \alpha_k^*\phi_{\theta_k^*}$.
\EndFor
\end{algorithmic}
\end{algorithm}

\subsection{Local error analysis}

The attractive features of Algorithm~\ref{alg:greed_ML} are
threefold. First, the algorithm simplifies the ML density estimation
procedure. Instead of facing the $N_{\rmp}$-mixture estimation
problem we only have to solve $N_{\rmp}$ $2$-mixture
estimation problems \cite{Li99}. Second, there are several bounds on
approximation and sampling errors of Algorithm~\ref{alg:greed_ML} in
terms of KL-divergence (see \cite{Li99} and \cite{Rak05}). In this
section we extend the existing results and perform the $L_p$ error
analysis. Third, it was shown \cite{Li99} that the performance of
the greedy algorithm converges to the performance of the optimal
mixture estimation algorithm as $N$ and $N_{\rmp}$ become
large.\vadjust{\goodbreak}

Here we state the relevant results from \cite{Li99} that will be of
use in further analysis. The following notation is introduced to
facilitate presentation. Assuming that $f$ is a target density and
$g \in\mathcal{C}$ we denote $D(f \| \mathcal{C}) = \inf_{g \in
\mathcal{C}} D(f \| g)$, the least possible divergence (bias)
between a target density,~$f$, and a member $g$ from the class of
continuous convex combinations~$\mathcal{C}$. Furthermore, assuming
that the target density $f$ is known, the analytical estimator
$g^{N_{\rmp}} \in\mathcal{C}_{N_{\rmp}}$ can be
obtained by solving the following greedy recursion for $i=2,\ldots,
N_{\rmp}$ (see Algorithm~\ref{alg:greed_ML}):
\[
(\theta_k^*,
\alpha_k^*) = \arg\max_{\alpha_k, \theta_k} \int
\log\bigl((1-\alpha_k)g_{k-1}(x) + \alpha_k\phi_{\theta_k}(x)\bigr) f(x) \,\rd
x.
\]
Alternatively, $\widehat g^{N_{\rmp}} \in
\mathcal{C}_{N_{\rmp}}$ is an empirical
$N_{\rmp}$-mixture estimator constructed using
Algorithm~\ref{alg:greed_ML} based on a sample from the target
density, $f$.

The following theorem (see \cite{Li99}) reveals an important general
property of the GML algorithm. It bounds the divergence between the
target density and the analytical estimator $g^{N_{\rmp}}$. The
bound is the sum of two terms. The first is the divergence between the
target density and an arbitrary approximating density
$g_{\mathcal{C}}\in\mathcal{C}$. The second term involves $\gamma$,
the upper bound on the log-ratio of two arbitrary functions from class
$\mathcal{C}$, and $c^2_{f,\mathcal{C}}$, a class dependent constant
(see \cite{Li99} for more detail). For example, for the class of
densities bounded below by $a$ and above by $b$ we have
$c^2_{f,\mathcal{C}} \leq(b/a)^2$. This second term features~$N_\mathrm{p}$ as
a denominator, so it tends toward zero as the number of components in
the mixture grows.

\begin{theorem}[(Li and Barron \cite{Li99}, Theorem 2)]
\label{th:li_barron} For every $g_{\mathcal{C}}(x) \in\mathcal{C}$
\[
D( f \| g^{N_{\rmp}} ) \leq D( f \| g_{\mathcal{C}} ) +
\frac{\gamma c^2_{f,\mathcal{C}}}{N_{\rmp}}.
\]
Here,
\[
c^2_{f,\mathcal{C}} = \int\frac{\int_\Theta\phi_{\theta}^2(x)
\mathbb{P}(\mathrm{d}\theta)}{(\int_\Theta\phi_{\theta}(x)
\mathbb{P}(\mathrm{d}\theta))^2} f(x) \,\mathrm{d}x,
\]
and $\gamma= 4[\log(3\sqrt{e}) + \sup_{\theta_1,\theta_2 \in
\Theta, x\in E} \log(\phi_{\theta_1}(x) / \phi_{\theta_2}(x)) ]$.
\end{theorem}

One of the consequences of Theorem~\ref{th:li_barron} is the
following relationship between an arbitrary $g_{\mathcal{C}}(x) \in
\mathcal{C}$ and the empirical GML algorithm output $\widehat
g^{N_{\rmp}} \in\mathcal{C}_{N_{\rmp}}$ \cite{Li99}:
%
%
\begin{equation} \label{eqn:li_barron_main}
\frac{1}{N}\sum_{i=1}^N \log\widehat
g^{N_{\rmp}}(\xi_i) \geq\frac{1}{N}\sum_{i=1}^N \log
g_\mathcal{C}(\xi_i) - \frac{\gamma
c^2_{f,\mathcal{C}}}{N_{\rmp}}.
\end{equation}
Clearly, it also follows directly from Theorem~\ref{th:li_barron}
that $D( f \| g^{N_{\rmp}} ) \leq D( f \| \mathcal{C} ) +
\frac{\gamma c^2_{f,\mathcal{C}}}{N_{\rmp}}$. Thus
Theorem~\ref{th:li_barron} establishes a strong formal argument that
shows that the greedy density estimate converges to the best
possible estimate as $N_{\rmp}$ grows without bound.

Our next goal is to connect the existing results on the performance of
the GML in terms of the KL-divergence to its performance in terms of
$L_p$ error metric. Our next result reveals the $L_p$ error bound
characterizing the average performance of the GML algorithm. The bound
consists of two components which arise because we split the total error
into approximation error (the
distance between the best analytical distribution $g^{N_\mathrm{p}}$ and $f$)
and sampling error (the additional error arising because the empirical
estimator $\widehat{g}^{N_\mathrm{p}}$ is derived from a sample from $f$,
rather than $f$ itself). The approximation error bound follows
directly from Theorem~\ref{th:li_barron}.

The bound on the sampling error is expressed in terms of the packing number
$\mathcal{D}(\varepsilon, \mathcal{H}, d_N)$, which is the maximum
number of $\varepsilon$-separated points in $\mathcal{H}$ (the class
of parametric density functions) and the entropy integral
\[
\int_{0}^b \sqrt{\log\bigl(1 + \mathcal{D}(\varepsilon, \mathcal{H},
d_N) \bigr)} \,\rd\varepsilon,
\]
both defined with respect to the empirical semimetric $d_{N}$,
which, in its turn, is defined for $h_1, h_2 \in\mathcal{H}$ as
follows:
\[
d_{N}^2(h_1, h_2) = \frac{1}{N}\sum_{k=1}^N \bigl(h_{1}(\xi_k) -
h_{2}(\xi_k)\bigr)^2.
\]
Examples of classes of functions with converging entropy integral
can be found in \cite{Rak05} and \cite{Vaa96}.\vspace*{-1.5pt}

\begin{theorem} \label{th:DLp}
Suppose $\widehat g^{N_{\rmp}} \in
\mathcal{C}_{N_{\rmp}}$ is constructed using
Algorithm~\ref{alg:greed_ML} and
$\widehat{\mathcal{G}}^{N_{\rmp}} \in\mathcal{P}(E)$ is the
distribution associated with $\widehat g^{N_{\rmp}}$. Suppose
further that there exists density $f$ associated with the target
distribution $F \in\mathcal{P}(E)$. Then for any $h \in
\mathcal{B}_\mathrm{b}(E)$ with $\|h\|_{\operatorname{osc}} \leq1$, $p \geq
1$, and $N,
N_{\rmp} \in\mathbb{N}$ we have
\begin{eqnarray*}
&&\mathbb{E} \{ |[\widehat{\mathcal{G}}^{N_{\rmp}} - F](h)
|^{p} \}^{1/p}\\[-2pt]
 && \qquad \leq\sqrt{2} \biggl[\frac{8}{a\sqrt{N}} \biggl(2
c^{2/p}(p/2)
+ (p/4)!C \bbE\int_{0}^b \sqrt{\log\bigl(1 +
\mathcal{D}(\varepsilon, \mathcal{H}, d_N) \bigr)} \,\rd\varepsilon
\biggr)\\[-2pt]
&&\hspace*{244pt} {}+ \frac{\gamma c^2_{f,\mathcal{C}}}{{N_{\rmp}}} + D(f
\| \mathcal{C}) \biggr]^{1/2},
\end{eqnarray*}
where $c(p)=1 \mbox{ if } 1/2\leq p<1$ and $C$ is a universal
constant.\footnote{See \cite{Vaa96} for details.}\vspace*{-1.5pt}
\end{theorem}

The following corollary addresses the special case when the target
density $f$ lies within the class of continuous convex
combinations,
$\mathcal{C}$. In this case, the approximation error bound approaches
$0$ as the number of mixture components grows.\vadjust{\goodbreak}

\begin{corollary} \label{cor:DLp}
Suppose that the assumptions of Theorem~\ref{th:DLp} hold. Suppose
in addition that $f \in\mathcal{C}$ then we have for any $p \geq
1$
\begin{eqnarray*}
&&\mathbb{E} \{ |[\widehat{\mathcal{G}}^{N_{\rmp}} - F](h)
|^{p} \}^{1/p}\\
 && \qquad \leq\sqrt{2} \biggl[\frac{8}{a\sqrt{N}} \biggl(2
c^{2/p}(p/2)
 + (p/4)!C \bbE\int_{0}^b \sqrt{\log\bigl(1 +
\mathcal{D}(\varepsilon, \mathcal{H}, d_N) \bigr)} \,\rd\varepsilon
\biggr)\\
&&\hspace*{214.5pt}
{} + 4\log\bigl(3\sqrt{e}(b/a)\bigr)\frac{(b/a)^2}{{N_{\rmp}}}
\biggr]^{1/2}.
\end{eqnarray*}
\end{corollary}

\begin{pf}
The proof follows from the fact that under the additional\vspace*{1pt} assumption
we have $D(f \| \mathcal{C})=0$. Furthermore, we note that under
this assumption $c^2_{f,\mathcal{C}} \leq(b/a)^2$ and $\gamma=
4\log(3\sqrt{e}(b/a))$.
\end{pf}

\subsection{Time-uniform error bounds}
\label{ssec:parametric_time_uni}

In this section we present a result specifying time-uniform error
bounds for the leader node particle filter performing parametric
approximation. The result links the properties of Markov transitions
$M_{i}$ and error bounds for parametric GML approximation
(Theorem~\ref{th:DLp}) with the propagation of approximation errors
through Feynman--Kac operators. It is based on the following
observations.

In the context of the GML algorithm and the leader node
recursion~\eqref{eqn:distr_PF_PA} the operator
$\mathbb{W}_{N_{\rmp}}$ can be described as follows:
\[
\frac{\rd}{\rd x_{i+1}} \mathbb{W}_{N_{\rmp}} \circ S^N
(\Phi_{i+1}^{\prime}(\widehat\eta_{i}^{\prime})) =
\sum_{j=1}^{N_{\rmp}} \alpha_{i,j} \phi_{\theta_{i,j}}
.
\]
This means that in this context our target density is
$\frac{\rd}{\rd
x_{i+1}}\Phi_{i+1}^{\prime}(\widehat\eta_{i}^{\prime})$ and we
obtain an i.i.d. sample from this density through the particle
filtering step. Based on the i.i.d. sample we estimate the weights,
$\alpha_{i,j}$, and parameters, $\theta_{i,j}$, of a mixture using
Algorithm~\ref{alg:greed_ML}. In the following we study the
conditions for the unbiased estimation of our target density and
then formulate our main result for the parametric approximation
leader node particle filter.

Suppose we can write the Markov kernel $M_{i}$ via its density
function $p_i(x_{i} | x_{i-1})$ \cite{DelMor04}
\[
M_{i}(x_{i-1}, \mathrm{d}x_{i}) = \Pr\{ X_i \in\mathrm{d}x_{i} |
X_{i-1} = x_{i-1} \} = p_{i}(x_{i} | x_{i-1}) \,\mathrm{d}x_{i} =
p_{\vartheta_{i}}(x_{i}) \,\mathrm{d}x_{i},
\]
where we explicitly assume that the structure of the kernel $M_{i}$
can be captured by a set of parameters $\vartheta_{i} \in
\varTheta_{i} \subset\mathbb{R}^{d_{i}}$ (these parameters may
include the state-value $x_{i-1}$). We can further define a class
$\mathcal{M}_{i}$ of such densities
\[
\mathcal{M}_{i} = \{ p_{\vartheta_{i}}(x_{i}) \dvtx \vartheta_{i}
\in\varTheta_{i} \subset\mathbb{R}^{d_{i}} \}.
\]
Furthermore, using the definitions of the one-step Boltzmann--Gibbs
transformation and the associated Feynman--Kac operator we see that\vadjust{\goodbreak}
the distribution at time $i+1$ is related to the distribution at
time $i$ as follows:\vspace*{-1pt}
\[
\frac{\rd\Phi_{i+1}^{\prime}(\widehat\eta_{i}^{\prime})}{\rd
x_{i+1}} = \int p_{i+1}(x_{i+1}|x_{i})
\frac{G_{i}^{\mathcal{S}_{\ell_{i}^{\prime}}}(x_{i})}{\eta
_{i}^{\prime}(G_{i}^{\mathcal{S}_{\ell_{i}^{\prime}}})}
\,\rd\widehat\eta_{i}^{\prime}.\vspace*{-1pt}
\]
Thus for a Markov kernel with $p_{\vartheta_{i+1}}(x_{i+1}) \in
\mathcal{M}_{i+1}$ we can rewrite the previous equation with a
suitable change of measure\vspace*{-1pt}
\[
\frac{\rd\Phi_{i+1}^{\prime}(\widehat\eta_{i}^{\prime})}{\rd
x_{i+1}} = \int_{\varTheta_{i+1}} p_{\vartheta_{i+1}}(x_{i+1})
\mathbb{P}(\mathrm{d} \vartheta_{i+1}).\vspace*{-1pt}
\]
This implies that for an $N$-particle approximation
$\widehat\eta_{i}^{\prime}$ [see \eqref{eqn:distr_PF_PA}] we\vspace*{-1pt} have
that $\frac{\rd\Phi_{i+1}^{\prime}(\widehat\eta_{i}^{\prime
})}{\rd
x_{i+1}} \in\conv_{N}(\mathcal{M}_{i+1})$ and, as $N$ grows without
bound, we have $\frac{\rd
\Phi_{i+1}^{\prime}(\widehat\eta_{i}^{\prime})}{\rd x_{i+1}} \in
\conv(\mathcal{M}_{i+1})$. Therefore the bias of the GML algorithm
in the leader node particle filter setting is determined by the
properties of Markov transition kernel $M_{i+1}$ and the class of
approximating densities $\mathcal{H}_{i+1}$. In particular, for the
Markov kernel with $p_{\vartheta_{i+1}}(x_{i+1}) \in
\mathcal{M}_{i+1}$ and a sufficiently rich class
$\mathcal{H}_{i+1}$, such that $\mathcal{M}_{i+1} \subseteq
\mathcal{H}_{i+1}$ we have asymptotically unbiased approximation
[recall that $\mathcal{C}_{i+1} = \conv(\mathcal{H}_{i+1})$]\vspace*{-1pt}
\[
D \biggl(\frac{\rd\Phi_{i+1}^{\prime}(\widehat\eta_{i}^{\prime
})}{\rd
x_{i+1}} \Big\| \mathcal{C}_{i+1} \biggr) = 0.\vspace*{-1pt}
\]
The preceding discussion can be summarized in the form of a concise
assumption:
\begin{itemize}
\item$(\mathcal{H})_{\rmu}$: The Markov kernels associated
with the target dynamics can be expressed in the form
$M_{i}(x_{i-1}, \mathrm{d}x_{i}) =
p_{\vartheta_{i}}(x_{i}) \,\mathrm{d}x_{i}$. The class of densities
associated with $M_{i}$ is defined as $\mathcal{M}_{i} = \{
p_{\vartheta_{i}}(x_{i}) \dvtx \vartheta_{i} \in\varTheta_{i} \subset
\mathbb{R}^{d_{i}} \}.$ Algorithm~\ref{alg:greed_ML} exploits
such classes $\mathcal{H}_i$
that there exist strictly positive numbers $a_{\rmu} =
\inf_{i \geq0} a_i$, $b_{\rmu} = \sup_{i \geq0} b_i$
satisfying $0 < a_{\rmu} <
b_{\rmu} < \infty$ and for any $i \geq0$ we have\vspace*{-1pt}
\[
\mathcal{M}_{i} \subseteq\mathcal{H}_{i}.\vspace*{-1pt}
\]
\end{itemize}

The following result describes the analog of
Theorem~\ref{th:Lp_time_uni} for the case of a parametric
approximation particle filter using the GML algorithm.\vspace*{-3pt}

\begin{theorem} \label{th:Lp_param_time_uni}
Suppose $\widehat\eta_t^{\prime}$ is defined
by~\eqref{eqn:distr_PF_PA} and assumptions $(G)_{\rmu}$,
$(M)_{\rmu}^{(m)}$ and $(\mathcal{H})_{\rmu}$ hold.
Suppose further that $\mathbb{P}\{ \Delta^{\prime}_{i} = 1\} \leq
q_{\rmu}$ for any $i \geq0$ and $0\leq q_{\rmu} \leq
1$. Then for any ${N_{\rmp}}, N \geq1$, $t \geq0$, $p \geq
1$ and $h_t \in\Osc_1(E_t)$ we have the time uniform bound
\begin{eqnarray*}
&&\hspace*{7pt} \sup_{t \geq0}\mathbb{E} \{ |[\widehat\eta_t^{\prime
} -
\eta_t^{\prime} ](h_t)|^p \}^{1/p}\\
 &&\hspace*{7pt}  \qquad \leq
\epsilon_{\rmu,m} \biggl[ \frac{c^{1/p}(p)}{\sqrt{N}}  +
q_{\rmu}^{1/p} \biggl[\frac{16}{a_{\rmu}\sqrt{N}} \biggl(2
c^{2/p}(p/2)+ C(p/4)! \\
&&\hspace*{7pt} \hphantom{\epsilon_{\rmu,m} \biggl[{} +
q_{\rmu}^{1/p} \biggl[\frac{16}{a_{\rmu}\sqrt{N}} \biggl(}
\qquad  \quad {} \times \sup_{i\geq0}\bbE\int_{0}^{b_i}
\log^{ {1}/{2}} \bigl(1 + \mathcal{D}(\varepsilon, \mathcal{H}_i,
d_N) \bigr) \,\rd\varepsilon\biggr)\\
&&\hspace*{7pt}\hspace*{214pt} {}+ 8\log\biggl(\frac{3\sqrt{e}
b_{\rmu}}{a_{\rmu}}\biggr)\frac{b_{\rmu}^2}{a_{\rmu}^2{N_{\rmp}}}
\biggr]^{1/2} \biggr].
\end{eqnarray*}
\end{theorem}

The above theorem provides an error bound for the parametric
approximation particle filter (using the GML algorithm to perform
approximation) that is similar in structure to that specified for the
subsampling approximation particle filter. The error bound consists of
two distinct contributions, one ($\frac{c^{1/p}(p)}{\sqrt{N}}$)
corresponding to the normal operation of the filter and the other
capturing the impact of the additional parametric approximation. The
bound on this second contribution is derived directly from the bound
expressed in Corollary~\ref{cor:DLp}. The theorem establishes a
requirement on the
sequence of approximating classes $\mathcal{H}_{i}$ leading to
unbiased approximation of distribution flows. The requirement is that
the Markov transition kernel must have an associated bounded density
and this density must be a member of the class $\mathcal{H}_{i}$. This
condition is reminiscent of the modeling assumptions that underpin
Gaussian sum particle filtering (see, e.g., \cite{Kot03}), where the
premise is that the filtering density can asymptotically be
represented as an infinite sum of Gaussians.

\section{Numerical experiments}
\label{sec:experiments}

In this section we present the results of numerical experiments
exploring the performance of the leader node particle filter. The
experiments provide an example of how the subsampling and parametric
approximation particle filters can be applied in a practical
tracking problem. They provide an opportunity to compare the
performance of the two algorithms and to examine whether practical
behavior is similar to that predicted by the theoretical analysis.

We adopt the following information acquisition and target movement
models. The state of the target is two-dimensional with
dynamics \cite{Ihl05}
\[
X_{t} = X_{t-1} + r_0([\cos\varphi_t; \sin\varphi_t]) + v_t.
\]
Here $r_0$ is a constant (we set $r_0 = 0.02$), and $\varphi_t, v_t$
are independent and uniformly distributed $v_t \sim U[0, 1]$,
$\varphi_t \sim U[-\pi, \pi]$. $K_l = 20$ leader nodes and $K_s =
200$ satellite nodes are distributed uniformly in the unit square. A~%
satellite sensor node $j$ with coordinates $s_j = [s_{1,j},
s_{2,j}]$ can transmit its measurement to any active leader node
within the connectivity radius $r_c$. The connectivity radius is set
to $r_c = \sqrt{2\log(K_s)/K_s}$. We assume that any active leader
node can route an approximation of its posterior representation to
any other potential leader node.

The measurement equation of every satellite sensor is the
binary detector~\cite{Coa05} capable\vadjust{\goodbreak} of detecting a target within
radius $r_d$ with probability $p_d$ and false alarm rate $p_f$
\[
\bbP\{ Y_t^j = 1 | X_t \} =
\cases{\displaystyle p_d ,&\quad if $X_t \in\mathcal{X}_d^j $,\cr
\displaystyle
p_f ,&\quad if $X_t \notin\mathcal{X}_d^j$,
}
\]
where the detection region $\mathcal{X}_d^j$ of satellite sensor $j$ is
defined as $\mathcal{X}_d^j = \{ x \dvtx \| x - s_j \|_2 \leq r_d \}$.
To perform sensor selection step we use the mutual information~(MI)
criterion \cite{Zhao03}
%
%
\begin{equation} \label{eqn:mutual_information}
\ell_{t+1} = \arg\max_{\ell_{t+1} \in\mathcal{L}} I(X_{t+1},
Y_{t+1}^{\mathcal{S}_{\ell_{t+1}}} |
y_{1}^{\mathcal{S}_{\ell_{1}}},\ldots,y_{t}^{\mathcal{S}_{\ell_{t}}}).
\end{equation}
Here the mutual information is defined as
\begin{eqnarray*}
&&I(X_{t+1}, Y_{t+1}^{\mathcal{S}_{\ell_{t+1}}} |
y_{1}^{\mathcal{S}_{\ell_{1}}},\ldots,y_{t}^{\mathcal{S}_{\ell_{t}}})\\
&& \qquad = \int p(x_{t+1},y_{t+1}^{\mathcal{S}_{\ell_{t+1}}}|
y_{1}^{\mathcal{S}_{\ell_{1}}},\ldots,y_{t}^{\mathcal{S}_{\ell_{t}}})
\\
&& \qquad  \quad  \hphantom{\int}
{}\times\log\biggl( \frac{p(x_{t+1},y_{t+1}^{\mathcal{S}_{\ell_{t+1}}}|
y_{1}^{\mathcal{S}_{\ell_{1}}},\ldots,y_{t}^{\mathcal{S}_{\ell
_{t}}})}{p(x_{t+1}|
y_{1}^{\mathcal{S}_{\ell_{1}}},\ldots,y_{t}^{\mathcal{S}_{\ell_{t}}})
p(y_{t+1}^{\mathcal{S}_{\ell_{t+1}}}|
y_{1}^{\mathcal{S}_{\ell_{1}}},\ldots,y_{t}^{\mathcal{S}_{\ell_{t}}})}
\biggr) \,\rd x\,\rd y,
\end{eqnarray*}
$y_{1}^{\mathcal{S}_{\ell_{1}}},\ldots,y_{t}^{\mathcal{S}_{\ell_{t}}}$
denotes the entire history of measurements, and the random variable
$Y_{t+1}^{\mathcal{S}_{\ell_{t+1}}}$ denotes the (potential) set of
measurements at time $t+1$ by the set of satellite sensor nodes
($\mathcal{S}_{\ell_{t+1}}$) of a candidate leader node
$\ell_{t+1}$.

Williams et al. pointed out in \cite{Wil07} that the application of
the one-step mutual information criterion for sensor selection can
result in undesirable leader node bouncing (frequent, unnecessary
hand-off). To prevent this, Williams et al. proposed a finite-time
horizon dynamic program \cite{Wil07}. In our simulations we use a
simpler randomized algorithm to control the leader node exchange
rate. In this algorithm the current leader node flips a biased coin
with the probability of the flip outcome being 1 equal to $\lambda$.
If the outcome is~1 then the current leader node calculates the
mutual information criterion. It then
applies~\eqref{eqn:mutual_information} to determine if the current
particle representation should be transferred to a new leader node
that is more likely to make informative measurements. The leader
node exchange rule can then be represented as
\[
\Delta_{t+1} =
\cases{\displaystyle1 ,&\quad if $(u_{t} \geq\lambda)$ and $(\ell
_{t+1}\neq\ell_{t})$,\cr\displaystyle
0 ,&\quad otherwise,
}
\]
where $u_{t}$ is uniformly distributed in $[0,1]$. With this
approach, the computational load for each leader node is
significantly reduced, and the communication overhead can be
regulated by the choice of $\lambda$. However, the value of
$\lambda$ should be tailored depending on the application. In our
experiments we fix $\lambda= 1/5$. Note that equations for
$\Delta_{t+1}$ and $\ell_{t+1}$ define the structure of an example
sensor management rule $\Upsilon_{t}^{\ell_t}$, which was formulated
in a more general form in Section~\ref{sec:Introduction}.

We consider two leader node particle filtering algorithms, with one
employing nonparametric approximation (subsampling) and the other
using parametric approximation. To create a subsample for
transmission in the nonparametric framework we use the general
residual resampling scheme \cite{Douc05}. The parametric leader node
particle filter is implemented using the GML algorithm with
$N_{\rmp}$ components. Each component consists of a
two-dimensional Gaussian density with diagonal covariance matrix.
The mean vector and covariance matrix are estimated using the
particle representation available at the current leader node. To
implement the GML algorithm we used the standard MATLAB nonlinear
optimization routine \texttt{fmincon} (see \cite{Ore09} for details of
the implementation).

%
\begin{figure}

\includegraphics{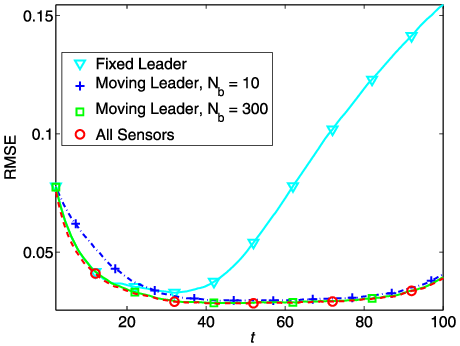}

\caption{Performance (RMSE) of different fusion schemes versus time:
$\bigtriangledown$ denotes the scheme with fixed leader node
selected at initialization; $+$ denotes the scheme with leader node
selected using approximate Mutual Information (MI) criterion and
nonparametric (subsampling) approximation with $N_{\rmb}=10$;
$\square$ denotes the scheme with leader node selected using
approximate MI criterion but no subsampling approximation
($N_{\rmb}=300$); and $\circ$ denotes the centralized scheme using
the entire set of measurements from all sensors.}
\label{fig:RMSE_fixed} 
\end{figure}

In the following we report the simulation results obtained using the
set-up discussed above. All results are achieved using 5,000 Monte
Carlo trials, and in each trial a new trajectory of the target is
generated.

Figure~\ref{fig:RMSE_fixed} depicts the performance in terms of Root
Mean Squared Error (RMSE) between the true position of the target
and its estimate using different information diffusion schemes. The
first scheme denoted by $\bigtriangledown$ corresponds to the
situation when the leader node is selected at the initialization and
is fixed throughout the tracking exercise. The second and third
schemes denoted by $+$ and $\square$ respectively correspond to
nonparametric leader node algorithms using $N_{\rmb}=10$ and
$N_{\rmb}=300$ particles for communications, respectively. The
fourth scheme denoted by $\circ$ corresponds to the centralized
scenario when all the measurements available from every sensor at
every time step $t$ are used to track the target. Note that the
baseline particle filter uses $N = 300$ particles (this value
was selected after experimentation with multiple values of $N$
because it provides sufficient accuracy without inducing unnecessary
computational overhead) in all scenarios (so
the $N_{\rmb}=300$ case corresponds to no subsampling). We can
see from Figure~\ref{fig:RMSE_fixed} that the centralized scheme is
only marginally better than the leader node scenario without
compression ($N=N_{\rmb}=300$). This confirms that our leader
node selection based on the approximate mutual information is a
valid approach.

The leader node particle filter that uses a very small number of
transmitted particles ($N_{\rmb}=10$) performs comparably
well. This suggests that there are practical scenarios where a
particle filter can incorporate aggressive approximation to reduce
communication overhead without incurring a~significant penalty in
tracking accuracy. The fixed leader node approach performs poorly
because the activated sensors only provide useful information when
the target is nearby.

In the next set of results, we explore the approximation error, that is,
the error induced by both sampling and the additional
parametric/subsampling approximations. The RMSE combines both
approximation error and estimation error resulting from the
inaccuracy and/or ambiguity of the measurement information. We can
estimate a \textit{Root Mean Squared Approximation Error} (RMSAE) by
calculating the error between a candidate particle filter and an
``ideal'' reference particle filter. As our reference filter, we
employ a~particle filter that uses $N=3\mbox{,}000$ particles, with no
approximation during hand-off. For each of the 5,000 Monte Carlo
trials, we apply this reference filter to generate location
estimates. The approximation error for our test filters is measured
relative to these estimates rather than the true locations.

Figure~\ref{fig:ratio_BP} depicts how the approximation performance
is affected as the number of particles in the subsampling step
($N_{\rmb}$) changes and the number of components in the
mixture model ($N_\mathrm{p}$) is varied. The performance is measured in
terms of the RMSAE increase relative to a leader node particle
filter that performs no additional approximation.

%
\begin{figure*}[t]
\centering
\begin{tabular}{@{}c@{\hspace*{10pt}}c@{}}

\includegraphics{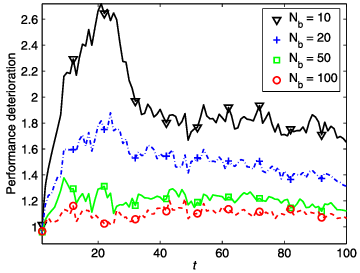}
&{\includegraphics{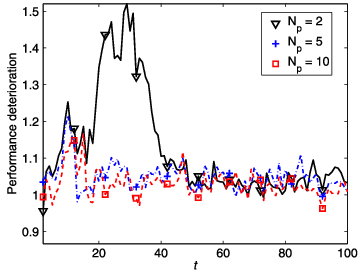}}\\
(a) &(b)
\end{tabular}
\caption{Deterioration of performance as a function of (\textup{a}) varying
number of transmitted particles and (\textup{b}) varying number of
transmitted mixture components for the leader node particle filter.
The performance deterioration is measured as the ratio of the RMSAE
of a leader node particle filter (with subsampling or parametric
approximation) to that of a~leader node particle filter with no
approximation ($N_{\mathrm{b}}=300$).
(\textup{a}) Nonparametric leader node particle filter.
(\textup{b}) Parametric leader node particle
filter.}
\label{fig:ratio_BP} 
\end{figure*}

Figure~\ref{fig:ratio_BP} indicates that the performance of the leader
node particle filter has interesting dynamic structure. In
particular, in the time period $t \in[1, 50]$ we can see an
articulated transient behavior [see Figure~\ref{fig:ratio_BP}(a),
$N_{\rmb}=10$ in particular]. The transient in these curves
arises because the particle representation of the target location
density is initially highly dispersed and multi-modal. However, as
time progresses ($t \in[51, 100]$) the particle representation of
the target becomes more localized and closer to unimodal, so
approximation performance improves  significantly. Qualitatively, the
performance deteriorates gracefully with respect to the extent of
the compression during hand-off (reduction in number of particles or
components), as theoretically predicted in the previous sections.

For the final performance analysis, we define a \textit{compression
factor} as the ratio of the number of particles used during regular
particle filter computations to the number of \textit{values}
transmitted during the hand-off. For the subsample approximation
case, this is simply $N/N_{\rmb}$. In our case of a Gaussian
mixture, variance information is transmitted in addition to the
locations of the Gaussians and the mixture weights, so the factor is
$2N/5N_{\rmp}$. Figure~\ref{fig:boxplot} presents a box-plot
depicting performance deterioration (ratio of approximation error of
the leader node with $N_\mathrm{b}<N$ and the leader node with $N_\mathrm{b}=N$)
versus the compression factor. Both the median and the maximal
deviations of the performance deterioration scale smoothly with
changing compression factor. Parametric approximation clearly
outperforms subsampling.

For the subsampling case, Corollary~\ref{cor:Lp_time_uni_tighter}
provides an analytical bound on the expected approximation error.
The curve based on this result [depicting the factor
$ (q_{\rmu} \chi+ (1-q_{\rmu}) )^{1/2}$;
experimentally measured $q_{\rmu}$ never exceeds $\lambda/2$]
is shown in Figure~\ref{fig:boxplot}(a) and provides a meaningful
characterization of the expected performance deterioration. For
comparison purposes, we include a~bound derived based on a simple
worst-case assumption that the subsample approximation particle
filter performs only as well as a particle filter that uses
$N_{\mathrm{b}}$ particles at all times. The bound developed in this
paper clearly provides a better indication of the performance
deterioration.

\section{Related work}
\label{sec:related}

In \cite{Kun71} Kunita studied the asymptotic behavior of the error
and stability of the filter that has an ergodic signal transition
semigroup with respect to the initial distribution. Ocone and
Pardoux \cite{Oco96} addressed the stability of linear filters with
respect to a non-Gaussian initial condition and examined the
stability of nonlinear filters in the case where the signal
diffusion is convergent. Although interesting, the results
in \cite{Kun71,Oco96} address the optimal filtering scenario, and
more relevant to our study is the analysis of approximately optimal
filters (especially particle filters). Important results concerning
the stability of particle filters have been developed over the past
decade \cite{Cri99,Mor01,Mor02,DelMor04,LeG98,Leg03,Gla04,Douc05}.

%
\begin{figure*}
\centering
\begin{tabular}{@{}c@{\hspace*{10pt}}c@{}}

\includegraphics{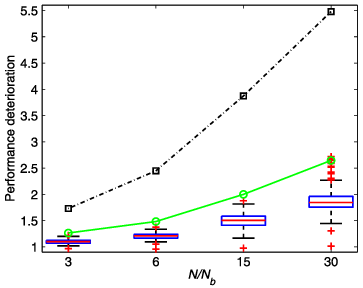}
&{\includegraphics{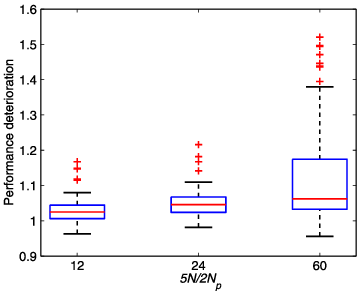}}\\
(a)
&(b)
\end{tabular}\vspace*{-3pt}
\caption{The relationship between deterioration of
approximation performance and compression factor. The performance
deterioration is measured as the ratio of the RMSAE of a~leader node
particle filter (subsampling or parametric approximation) to that of a~%
leader node particle filter
with no approximation ($N_{\mathrm{b}}=300$). The compression factor
is the ratio of $N$ to the
number of values transmitted during leader node exchange. The boxes
show lower quartile, median and upper quartile. Whiskers depict 1.5
times the
interquartile range, and the $+$ values denote outliers.
(\textup{a}) Nonparametric leader node particle filter.
$\square$ denotes the naive performance deterioration
characterization, $\sqrt{N/N_{\rmb}}$. $\circ$ denotes the
proposed characterization captured by
Corollary~\protect\ref{cor:Lp_time_uni_tighter}.
(\textup{b}) Parametric leader node particle filter.}
\label{fig:boxplot} 
\end{figure*}

The Feynman--Kac semigroup approach to the stability analysis of
particle filters has been described and developed by Del Moral, Miclo
and Guionnet in \cite{Mor01,Mor02,DelMor04}. The authors study the
stability properties of general nonlinear Feynman--Kac semigroups
under a variety of assumptions. The Dobrushin contraction coefficient
of the underlying Markov chain plays a central role in the
analysis. In \cite{Mor02}, Del~Moral and Miclo formulate the
conditions for the exponential asymptotic stability of the Feynman--Kac
semigroup and bound the Lyapunov constant and Dobrushin
coefficient. One of the applications of these results is a
time-uniform upper bound on the error of interacting particle
systems. In \cite{DelMor04}, Del Moral provides an extensive analysis of
the properties of Feynman--Kac semigroups.\vadjust{\goodbreak} His analysis forms the basis
for our study in this paper, particularly in the case of the
subsampling approximation particle filter.

Stability analysis for particle filters is frequently built on
relatively strong assumptions about the mixing and ergodicity
properties of the underlying Markov transitions of the signal
(target state). There have been some efforts to relax these types of
assumptions. In \cite{Leg03,Gla04}, Le Gland and Oudjane study the
stability and convergence rates for particle filters using the
Hilbert projective metric. In \cite{Leg03}, they relax the signal
mixing assumptions by employing a specific, ``robust'' particle
filter architecture with truncated likelihood functions.

In the subsampling approximation particle filter analyzed in this
paper, the number of particles varies over time. Crisan et al.
examine the stability of branching and interacting particle systems
in \cite{Cri99}; in these systems the population size also varies
because at each time step a particle generates a~random number of
offspring. The properties of the resulting particle filter depend on
the initial number of particles. The variation in the number of
particles is clearly very different from that of the subsampling
approximation particle filter, so the results are not directly
applicable.

Thus far we have discussed previous work that has addressed particle
filter stability when the error arises due to the sampling
approximation. The sampling error is dependent on the resampling
schemes, and Douc et al. have provided theoretical results that
allow various resampling schemes to be compared \cite{Douc05}.

Le Gland et al. provide uniform convergence results for the
regularized particle filters \cite{LeG98,Gla04}. Although there is
some similarity to the parametric approximation particle filter we
analyze, the purpose of the approximation is very different. It is
not performed intermittently to reduce computation or communication
cost, but rather is performed every time step with a complex model
($N$ components). From an algorithmic standpoint, there are also
similarities with the Gaussian sum particle filter \cite{Kot03}, but
the theoretical analysis of this filter is less developed.

There has been some work addressing the analysis of the leader node
particle filter \cite{Ihl05}. Although simulation (and to some
extent, experimental) results indicate that instability effects are
rarely observed in the leader node particle filtering, prior to our
work, the theoretical bounds on estimation error for leader node
particle filtering using intermittent parametric approximation grow
exponentially over time \cite{Ihl05}.

\section{Concluding remarks}
\label{section::Concluding Remarks}

We have presented the analysis of the leader node particle filter
that performs intermittent approximation. Our main results have the
form of upper bounds on the expected $L_p$ approximation error of
the leader node particle filter that occasionally employs either
subsampling or parametric approximations of the filtering
distribution. Such approximation steps become necessary when
particle filters are deployed on resource-constrained platforms,
where the resource can be energy, memory or computational power. The
important conclusion of our analysis is that these approximation
steps do not induce instability, and moreover, the frequency of the
approximation steps significantly affects the extent of performance
degradation. If the approximation steps are rare, then the
compression can be significant (a subset of subsamples or a few
mixture components are used during leader node exchange), and the
error remains reasonable. Numerical experiments indicate that the
bound for the subsample approximation particle filter provides a
meaningful characterization of practical approximation performance.

\section{Proofs of theorems} \label{sec:proofs}

\mbox{}

\begin{pf*}{Proof of Theorem~\ref{th:Lp_time_uni}}
We begin by applying Minkowski's inequality to
(\ref{eqn:err_decomp}):\vspace*{-2pt}
\[
\mathbb{E} \{ |[\widehat\eta_t^{\prime} -
\eta_t^{\prime}](h_t) |^p \}^{ {1}/{p}} \leq
\sum_{i=0}^t \mathbb{E} \{ |
[\Phi_{i,t}^{\prime}(\widehat\eta_i^{\prime}) -
\Phi_{i,t}^{\prime}(\Phi_{i}^{\prime}(\widehat\eta_{i-1}^{\prime}))
](h_t) |^p \}^{ {1}/{p}},
\]
and then using~\eqref{eqn:dobrushin_bound},
\eqref{eqn:dobrushin_onestep_bound} and \eqref{eqn:G_osc_bound} we
have
\begin{eqnarray*} 
&&\sum_{i=0}^t \mathbb{E} \{ |
[\Phi_{i,t}^{\prime}(\widehat\eta_i^{\prime}) -
\Phi_{i,t}^{\prime}(\Phi_{i}^{\prime}(\widehat\eta_{i-1}^{\prime}))
](h_i) |^p \}^{ {1}/{p}}\\[-1pt]
 && \qquad \leq\frac{2 -
\epsilon_{\rmu}(M) \epsilon_{\rmu}^{m
K_{\rmu}}(G)}{\epsilon_{\rmu}(M)
\epsilon_{\rmu}^{m K_{\rmu}}(G)} \\[-1pt]
&& \qquad  \quad {}\times\sum_{i=0}^t \bigl( 1 -
\epsilon_{\rmu}^{2}(M) \epsilon_{\rmu}^{(m-1)
K_{\rmu}}(G) \bigr)^{\lfloor(t-i)/m \rfloor} \mathbb{E}
\{
| [ \widehat\eta_i^{\prime} -
\Phi_{i}^{\prime}(\widehat\eta_{i-1}^{\prime}) ](h_i)
|^{p} \}^{1/p}.
\end{eqnarray*}
Next we analyze each individual expectation under the sum above. In
particular, using the structure of the algorithm defined in
(\ref{eqn:distr_PF_NP}) and the definition of sampling operator
introduced in (\ref{eqn:SN}) we can rewrite the terms comprising the
sum in the following explicit way:
%
%
\begin{eqnarray} \label{eqn:expectation_explicit}
&&\mathbb{E} \{ | [\widehat\eta_i^{\prime} -
\Phi_{i}^{\prime}(\widehat\eta_{i-1}^{\prime}) ](h_i)
|^p
\}^{ {1}/{p}} \nonumber\\
&&\qquad= \mathbb{E} \{ |
[\Delta^{\prime}_{i}S^{N}\circ
S^{{N_{\rmb}}}(\Phi_{i}^{\prime}(\widehat\eta_{i-1}^{\prime}))
 +
(1-\Delta^{\prime}_{i})S^{N}(\Phi_{i}^{\prime}(\widehat\eta
_{i-1}^{\prime}))
\\
&&\hspace*{214pt}{}- \Phi_{i}^{\prime}(\widehat\eta_{i-1}^{\prime}) ](h_i)
|^{p} \}^{{1}/{p}} \nonumber.
\end{eqnarray}
Grouping the terms and using Minkowski's inequality again, we
conclude
\begin{eqnarray*}
&&\mathbb{E} \{ | [\widehat\eta_i^{\prime} -
\Phi_{i}^{\prime}(\widehat\eta_{i-1}^{\prime})
](h_i) |^p \}^{ {1}/{p}} \\
&& \qquad\leq\mathbb{E} \{ | \Delta^{\prime}_{i} [
S^{N}\circ
S^{{N_{\rmb}}}(\Phi_{i}^{\prime}(\widehat\eta_{i-1}^{\prime}))
- \Phi_{i}^{\prime}(\widehat\eta_{i-1}^{\prime}) ](h_i)
|^{p}
\}^{ {1}/{p}} \\
&& \qquad \quad{}+ \mathbb{E} \{ |(1-\Delta^{\prime}_{i})
[S^{N}(\Phi_{i}^{\prime}(\widehat\eta_{i-1}^{\prime})) -
\Phi_{i}^{\prime}(\widehat\eta_{i-1}^{\prime}) ](h_i)
|^{p} \}^{ {1}/{p}}.
\end{eqnarray*}
Adding and subtracting $\Delta^{\prime}_i
S^{{N_{\rmb}}}(\Phi_{i}^{\prime}(\widehat\eta_{i-1}^{\prime}))$
in the first term, we have
%
%
\begin{eqnarray} \label{eqn:ith_term_error_decomp}
&&\mathbb{E} \{ | [\widehat\eta_i^{\prime} -
\Phi_{i}^{\prime}(\widehat\eta_{i-1}^{\prime}) ](h_i)
|^p
\}^{ {1}/{p}} \nonumber\\
&& \qquad \leq\mathbb{E} \{ | \Delta^{\prime}_{i} [
S^{N}\circ
S^{{N_{\rmb}}}(\Phi_{i}^{\prime}(\widehat\eta_{i-1}^{\prime}))
-
S^{{N_{\rmb}}}(\Phi_{i}^{\prime}(\widehat\eta_{i-1}^{\prime}))
](h_i) |^{p}
\}^{ {1}/{p}} \nonumber
\\[-8pt]
\\[-8pt]
&& \qquad \quad{} + \mathbb{E} \{ | \Delta^{\prime}_{i} [
S^{{N_{\rmb}}}(\Phi_{i}^{\prime}(\widehat\eta_{i-1}^{\prime}))
- \Phi_{i}^{\prime}(\widehat\eta_{i-1}^{\prime})
](h_i) |^{p} \}^{ {1}/{p}} \nonumber\\
&& \qquad \quad{} + \mathbb{E} \{ |(1-\Delta^{\prime}_{i})
[S^{N}(\Phi_{i}^{\prime}(\widehat\eta_{i-1}^{\prime})) -
\Phi_{i}^{\prime}(\widehat\eta_{i-1}^{\prime}) ](h_i)
|^{p} \}^{ {1}/{p}} .
\nonumber
\end{eqnarray}
We see that each error term under the sum splits into three
individual terms, describing the approximation paths the leader node
algorithm can follow at time $i$. If $N = \chi N_{\rmb}$, then
the $N$-particle approximation after subsampling can be recovered
from the $N_{\rmb}$-particle approximation without error by
replicating the $N_{\rmb}$-particle approximation $\chi$
times. Thus the first term in (\ref{eqn:ith_term_error_decomp}) is
zero.

The analysis of the remaining two terms is similar. We first
concentrate on the second term. Recall that $(\ell_{i}^{\prime},
\Delta_{i}^{\prime}) =
\Upsilon_{i-1}^{\ell_{i-1}^{\prime}}(\Phi_{i}^{\prime}(\widehat
\eta_{i-1}^{\ell_{0\dvtx i-1}^{\prime}}),
\mathcal{I}_{i-1})$. Thus given the $\sigma$-algebra
$\mathcal{F}_{i-1}$ and the realization of the measurement taken by
leader node $\ell_{i-1}^{\prime}$,
$Y_{i-1}^{\mathcal{S}_{\ell_{i-1}^{\prime}}} =
y_{i-1}^{\mathcal{S}_{\ell_{i-1}^{\prime}}}$, the output of the
decision rule is independent of the sampling error,
$[S^N(\Phi_{i}^{\prime}(\widehat\eta_{i-1}^{\prime})) -
\Phi_{i}^{\prime}(\widehat\eta_{i-1}^{\prime})](h_i)$. We exploit
this Markovian nature of the decision rule and apply
Lemma~\ref{lem:2} to the conditional expectation rendering the
following bound:
%
%
\begin{eqnarray} \label{eqn:conditioning_argument}
&&\mathbb{E} \{ | \Delta^{\prime}_{i} [
S^{{N_{\rmb}}}(\Phi_{i}^{\prime}(\widehat\eta_{i-1}^{\prime}))
- \Phi_{i}^{\prime}(\widehat\eta_{i-1}^{\prime})
](h_i) |^{p} \}^{1/p} \nonumber\\
&& \qquad= \mathbb{E} \bigl\{ \Delta^{\prime}_{i} \mathbb{E} \{ |
[
S^{{N_{\rmb}}}(\Phi_{i}^{\prime}(\widehat\eta_{i-1}^{\prime}))
- \Phi_{i}^{\prime}(\widehat\eta_{i-1}^{\prime}) ](h_i)
|^{p} | \mathcal{F}_{i-1} ,
Y_{i-1}^{\mathcal{S}_{\ell_{i-1}^{\prime}}} =
y_{i-1}^{\mathcal{S}_{\ell_{i-1}^{\prime}}}
\} \bigr\}^{1/p} \\
&& \qquad\leq\frac{c^{1/p}(p)}{\sqrt{{N_{\rmb}}}} q_{i}^{1/p}.
\nonumber
\end{eqnarray}
Combining the analysis results for all three terms, we obtain
\[
\mathbb{E} \{ | [\widehat\eta_i^{\prime} -
\Phi_{i}^{\prime}(\widehat\eta_{i-1}^{\prime}) ](h_t)
|^p \}^{1/p} \leq c^{1/p}(p) \biggl( q_{i}^{1/p}
\frac{1}{\sqrt{{N_{\rmb}}}} + (1 - q_{i})^{1/p}
\frac{1}{\sqrt{N}} \biggr).
\]
We note that the expression in brackets has the form $\varphi(q_{i})
= q_{i}^{1/p} (\alpha+\beta) + (1 - q_{i})^{1/p} \alpha$ for some
$\beta> \alpha\geq0$. For $p \geq1$, $\varphi(q_{i})$ has
maximum at $q_{i} = q_{\max}$,
\[
q_{\max} =
\frac{1}{1 + [ ({\alpha+\beta})/{\alpha} ]^{p/(1-p)}}.
\]
We have that $\varphi(q_{i})$ is nondecreasing on $q_{i} \in[0,
q_{\max}]$ and nonincreasing on $q_{i} \in
(q_{\max}, 1]$. Noting that
$ [(\alpha+\beta)/\alpha]^{p/(1-p)}$ is increasing in $p$
we obtain
\[
q_{\max} \geq\frac{1}{1 + [{\alpha}/({\alpha
+\beta}) ]}
\geq\inf_{\beta\dvtx \beta> \alpha} \frac{1}{1 +
[ {\alpha}/({\alpha+\beta}) ]} = 2/3.
\]
Thus if
$q_{\rmu} \leq2/3 \leq q_{\max}$, then for any $i \geq
0$ we have
\[
\mathbb{E} \{ | [\widehat\eta_i^{\prime} -
\Phi_{i}^{\prime}(\widehat\eta_{i-1}^{\prime}) ](h_t)
|^p \}^{1/p} \leq c^{1/p}(p) \biggl( q_{\rmu}^{1/p}
\frac{1}{\sqrt{{N_{\rmb}}}} + (1 - q_{\rmu})^{1/p}
\frac{1}{\sqrt{N}} \biggr).
\]
Finally, noting \cite{DelMor04} that
%
%
\begin{equation} \label{eqn:const_sum_bound}
\sum_{i=0}^t \bigl( 1 - \epsilon_{\rmu}^{2}(M)
\epsilon_{\rmu}^{(m-1) K_{\rmu}}(G) \bigr)^{\lfloor
(t-i)/m \rfloor} \leq\frac{m}{\epsilon_{\rmu}^{2}(M)
\epsilon_{\rmu}^{(m-1) K_{\rmu}}(G)},
\end{equation}
we complete the proof of theorem.
\end{pf*}

\begin{pf*}{Proof of Corollary \ref{cor:Lp_time_uni}}
The corollary follows by allowing for sampling error to arise in the
first term in
(\ref{eqn:ith_term_error_decomp}):
\[
\mathbb{E} \{ | \Delta^{\prime}_{i} [ S^{N}\circ
S^{{N_{\rmb}}}(\Phi_{i}^{\prime}(\widehat\eta_{i-1}^{\prime}))
-
S^{{N_{\rmb}}}(\Phi_{i}^{\prime}(\widehat\eta_{i-1}^{\prime}))
](h_i) |^{p} \}^{ {1}/{p}} \leq
\frac{c^{1/p}(p)}{\sqrt{N}} q_{i}^{1/p}
\]
and incorporating this error bound throughout the rest of the proof of
Theorem~\ref{th:Lp_time_uni}.
\end{pf*}

\begin{pf*}{Proof of Corollary \ref{cor:Lp_time_uni_tighter}}
Starting with (\ref{eqn:expectation_explicit}), we perform a
different error decomposition expanding the power. We observe that
$\Delta^{\prime}_{i}(1-\Delta^{\prime}_{i}) = 0$ and that if $N =
\chi N_{\rmb}$ for integer $\chi$, we can reconstruct an
$N$-sample representation from the $N_{\rmb}$ sample with no
additional error. Thus
\begin{eqnarray*}
&&\mathbb{E} \{ | [\widehat\eta_{i}^{\prime} -
\Phi_{i}^{\prime}(\widehat\eta_{i-1}^{\prime}) ](h_i)
|^p \}^{ {1}/{p}}\\ && \qquad   \leq\mathbb{E} \{
\Delta^{\prime}_{i}
| [S^{{N_{\rmb}}}(\Phi_{i}^{\prime}(\widehat\eta
_{i-1}^{\prime}))
-
\Phi_{i}^{\prime}(\widehat\eta_{i-1}^{\prime}) ](h_i)
|^p \\
&&\hphantom{\mathbb{E} \{} \qquad  \quad {}+ (1-\Delta^{\prime}_{i})
| [S^{N}(\Phi_{i}^{\prime}(\widehat\eta_{i-1}^{\prime
})) -
\Phi_{i}^{\prime}(\widehat\eta_{i-1}^{\prime}) ](h_i)
|^{p} \}^{1/p} .
\end{eqnarray*}
Applying the same conditioning as in
\eqref{eqn:conditioning_argument} and utilizing Lemma~\ref{lem:2},
%
%
\begin{equation}
\mathbb{E} \{ | [\widehat\eta_{i}^{\prime} -
\Phi_{i}^{\prime}(\widehat\eta_{i-1}^{\prime}) ](h_i)
|^p \}^{ {1}/{p}} \leq
\frac{c(p)^{1/p}}{\sqrt{N}} \bigl( q_i \chi^{p/2} + (1-q_i)
\bigr)^{1/p}.
\end{equation}
We note that $\chi\geq1$ and $q_i \chi+ (1-q_i) \leq q_{\rmu}
\chi+ (1-q_{\rmu})$ under the assumption $q_i \leq
q_{\rmu}$. The final step in the proof involves applying
(\ref{eqn:const_sum_bound}) as in the proof of
Theorem~\ref{th:Lp_time_uni}.
\end{pf*}

\begin{pf*}{Proof of Theorem~\ref{th:Exp_time_uni}}
Using the triangle inequality in (\ref{eqn:err_decomp}), following
the methodology presented in Theorem~\ref{th:Lp_time_uni} and
denoting $\omega_i = ( 1 - \epsilon_{\rmu}^{2}(M)\times\break
\epsilon_{\rmu}^{(m-1) }(G) )^{\lfloor(t-i)/m \rfloor}$
and $a = \frac{2 - \epsilon_{\rmu}(M) \epsilon_{\rmu}^{m
}(G)}{\epsilon_{\rmu}(M)
\epsilon_{\rmu}^{m }(G)}$ we have
\[
|[\widehat\eta_t^{\prime} - \eta_t^{\prime}](h_t) |
\leq
a \sum_{i=0}^t \omega_i | [\widehat\eta
_i^{\prime}
- \Phi_{i}^{\prime}(\widehat\eta_{i-1}^{\prime}) ](h_i)
|.
\]
Using the structure of the algorithm defined in
(\ref{eqn:distr_PF_NP}) and the definition of sampling operator\vadjust{\goodbreak}
introduced in (\ref{eqn:SN}), we obtain the following (similarly to
Theorem~\ref{th:Lp_time_uni}):
\[
|[\widehat\eta_t^{\prime} - \eta_t^{\prime}](h_t) |
\leq
Z_1 + Z_2,
\]
where
\begin{eqnarray*}
Z_1 &=& a \sum_{i=0}^t \omega_i \Delta^{\prime}_i |
[S^{N}\circ
S^{{N_{\rmb}}}(\Phi_{i}^{\prime}(\widehat\eta_{i-1}^{\prime}))
-
S^{{N_{\rmb}}}(\Phi_{i}^{\prime}(\widehat\eta_{i-1}^{\prime
})) ](h_i) | \\
&&{}+ a \sum_{i=0}^t \omega_i (1-\Delta^{\prime}_i) |
[S^{N}(\Phi_{i}^{\prime}(\widehat\eta_{i-1}^{\prime})) -
\Phi_{i}^{\prime}(\widehat\eta_{i-1}^{\prime}) ](h_i)
|,
\\
Z_2 &=& a\sum_{i=0}^t \omega_i \Delta^{\prime}_i |
[S^{{N_{\rmb}}}(\Phi_{i}^{\prime}(\widehat\eta
_{i-1}^{\prime}))
- \Phi_{i}^{\prime}(\widehat\eta_{i-1}^{\prime}) ](h_i)
|.
\end{eqnarray*}
Noting that
\[
\sup_{t \geq0} \mathbb{P} \{ |[\widehat\eta
_t^{\prime}
- \eta_t^{\prime}](h_t)| \geq\epsilon\} \leq\sup_{t
\geq0} \mathbb{P} \{ Z_1 + Z_2 \geq\epsilon\}
\]
and using the fact that for any couple of random variables $Z_1,Z_2$
we have $(Z_1 + Z_2 \geq1) \rightarrow((Z_1 \geq1/2) \mbox{ or }
(Z_2 \geq1/2))$ and thus $\bbP\{ X + Y \geq\varepsilon\} \leq
\bbP\{ X \geq\varepsilon/2 \} + \bbP\{ Y \geq\varepsilon/2 \}$, we
have
\[
\sup_{t \geq0} \mathbb{P} \{ |[\widehat\eta
_t^{\prime}
- \eta_t^{\prime}](h_t)| \geq\epsilon\} \leq\sup_{t
\geq0} \mathbb{P} \{ Z_1 \geq\epsilon/2 \} +
\sup_{t \geq0} \mathbb{P} \{ Z_2 \geq\epsilon/2
\}.
\]
Now applying Markov's inequality we conclude
\[
\sup_{t \geq0} \mathbb{P} \{ |[\widehat\eta
_t^{\prime}
- \eta_t^{\prime}](h_t)| \geq\epsilon\} \leq\sup_{t
\geq0} e^{-\tau_1\epsilon/2}\mathbb{E} \{ e^{\tau_1
Z_1} \} + \sup_{t \geq0} e^{-\tau_2\epsilon/2}
\mathbb{E} \{ e^{\tau_2 Z_2} \}.
\]
Next we apply the exponential series expansion,
%
%
\begin{equation} \label{eqn:exp_series_expansion}
\mathbb{E} \{ e^{\tau_1 Z_1} \} = \sum_{n\geq0}
\frac{\tau_1^n}{n!} \mathbb{E}\{Z_1^n\},
\end{equation}
and use the fact that according to the following conditioning
argument and Lemma~\ref{lem:2}, we have
\begin{eqnarray*}
\mathbb{E}\{Z_1^n\}^{1/n} &=& ( \bbE\{Z_1^n |\Delta^{\prime}_i =
1\}\bbP\{ \Delta^{\prime}_i = 1 \} + \bbE\{Z_1^n |\Delta^{\prime}_i
=
0\}\bbP\{ \Delta^{\prime}_i = 0 \} )^{1/n} \\
&\leq& a \sum_{i=0}^t \omega_i \bigl( q_i c(n) N^{-n/2} + (1-q_i)
c(n) N^{-n/2}\bigr)^{1/n} = \frac{c^{1/n}(n)}{\sqrt{N}}
a\sum_{i=0}^t \omega_i.
\end{eqnarray*}
Noting that $a \sum_{i=0}^t \omega_i \leq
\epsilon_{\rmu,m}$ we have that $\mathbb{E}\{Z_1^n\} \leq
\epsilon_{\rmu,m}^n c(n) N^{-n/2}$. Substituting this into
(\ref{eqn:exp_series_expansion}) and employing the same
simplifications as in the proofs of Theorem~\ref{th:3} and
Corollary~\ref{cor:simple_MGF_bound} we obtain
\[
e^{- \varepsilon\tau_1/2}\mathbb{E} \{ e^{\tau_1 Z_1} \}
\leq\biggl(1 + \sqrt{2\pi}\frac{\tau_1
\epsilon_{\rmu,m}}{\sqrt{N}} \biggr)
e^{ {\tau_1^2\epsilon_{\rmu,m}^2}/({8N}) - \varepsilon\tau_1/2}.
\]
Applying similar analysis to
$e^{-\varepsilon\tau_2/2}\mathbb{E} \{ e^{\tau_2 Z_2} \}$
and choosing\vspace*{-3pt} $\tau_1 = \frac{2\varepsilon
N}{\epsilon_{\rmu,m}^2}$ and $\tau_2 =
\frac{2N_{\rmb}\varepsilon}{\epsilon_{\rmu,m}^2}$ completes
the proof.
\end{pf*}

\begin{pf*}{Proof of Theorem~\ref{th:DLp}}
Using Pinsker's inequality, $\int|f\,{-}\,g|\,{\leq}\,\sqrt{2 D(f\|g)}$,
\cite{Dev00} we have
\[
\mathbb{E} \{ |[\widehat{\mathcal{G}}^{N_{\rmp}} - F](h)
|^p \}^{1/p} \leq\sqrt{2} [ \mathbb{E} \{ D(f \|
\widehat g^{N_{\rmp}}) ^{p/2} \}^{2/p} ]^{1/2}.
\]
Now, suppose $p \geq2$. The following decomposition can be used to
analyze the previous expression:
\[
D(f \| \widehat g^{N_{\rmp}}) = D(f \| \widehat
g^{N_{\rmp}}) - D(f \| \mathcal{C}) + D(f \| \mathcal{C}).
\]
Denoting $g^{*} = \arg\min_{g \in\mathcal{C}} D(f \| g)$ we have
the following modification of the decomposition proposed by Rakhlin et
al.
in \cite{Rak05}:
\begin{eqnarray*}
D(f \| \widehat g^{N_{\rmp}}) - D(f \| \mathcal{C}) &=& -\int
\log\widehat g^{N_{\rmp}}(x) F(\mathrm{d}x) +
\frac{1}{N}\sum_{i=1}^N \log\widehat
g^{N_{\rmp}}(\xi_i)
\\
&&{}+ \frac{1}{N}\sum_{i=1}^N \log g^{*}(\xi_i) -
\frac{1}{N}\sum_{i=1}^N \log\widehat g^{N_{\rmp}}(\xi
_i) \\
&&{}+ \int\log g^{*}(x) F(\mathrm{d}x) -
\frac{1}{N}\sum_{i=1}^N \log g^{*}(\xi_i).
\end{eqnarray*}
Applying (\ref{eqn:li_barron_main}) to the middle term we see
\begin{eqnarray*}
&&D(f \| \widehat g^{N_{\rmp}}) - D(f \| \mathcal{C}) \\
&& \qquad \leq|[F
- S^N(F)](\log\widehat g^{N_{\rmp}})| + |[F - S^N(F)](\log
g^{*})| + \frac{\gamma c^2_{f,\mathcal{C}}}{{N_{\rmp}}}.
\end{eqnarray*}
By the definition of $D(f \| \mathcal{C})$ it follows that $D(f \|
\widehat g^{N_{\rmp}}) - D(f \| \mathcal{C}) \geq0$, and thus
we conclude
\[
| D(f \| \widehat g^{N_{\rmp}}) - D(f \| \mathcal{C})
| \leq2 \sup_{g \in\mathcal{C}} |[F -
S^N(F)] (\log g ) | + \frac{\gamma
c^2_{f,\mathcal{C}}}{{N_{\rmp}}}.
\]
This allows us to split the effect of approximation and sampling
errors by applying Minkowski's inequality (since $p \geq2$),
\begin{eqnarray*}
\mathbb{E} \{ D(f \| \widehat g^{N_{\rmp}}) ^{p/2}
\}^{2/p} &\leq&2\mathbb{E} \Bigl\{ \Bigl[\sup_{g \in
\mathcal{C}} |[F - S^N(F)] (\log g ) |
\Bigr]^{p/2} \Bigr\}^{2/p}\\
&&{} + \frac{\gamma
c^2_{f,\mathcal{C}}}{{N_{\rmp}}} + D(f \| \mathcal{C}).
\end{eqnarray*}

The next step of the proof makes use of a symmetrization argument.
We recall the definition of the Rademacher sequence
$(\varepsilon_k)$ as a sequence of independent random variables
taking values in $\{-1, +1\}$ with $\bbP\{\varepsilon_k = 1\} =
\bbP\{\varepsilon_k = -1\} = 1/2$. Denote by $S_{\varepsilon}^{N}$
the generator of the signed Rademacher measure (with $\xi_k$ being
the samples from $\mu$)
\[
S_{\varepsilon}^{N}(\mu)(h) = \frac{1}{N}\sum_{k=1}^N \varepsilon_k
h(\xi_k).
\]

Using the symmetrization lemma (see, e.g., Lemma~2.3.1 in \cite{Vaa96}
or Lem\-ma~6.3 in \cite{Led91}), we deduce
\[
\mathbb{E} \Bigl\{ \Bigl[\sup_{g \in\mathcal{C}} |[F -
S^N(F)] (\log g ) | \Bigr]^{p/2} \Bigr\}^{2/p}
\leq
2 \mathbb{E}\Bigl \{\Bigl [\sup_{g \in\mathcal{C}}
|S^N_{\varepsilon}(F) (\log g ) | \Bigr]^{p/2}
\Bigr\}^{2/p}.
\]
Denoting $\kappa(x)\,{=}\,g(x)\,{-}\,1$ and using the fact \cite{Rak06} that
\mbox{$\varphi(\kappa(x))\,{=}\,a \log(\kappa(x)\,{+}\,1)$} is a
contraction,\footnote{The function $\varphi\dvtx \mathbb{R} \rightarrow
\mathbb{R}$ is a contraction if we have $|\varphi(x) - \varphi(y)|
\leq|x-y|, \forall x,y \in E$.} we apply the comparison inequality
(Theorem~4.12 in~\cite{Led91}), observing that $[\cdot]^{p/2}$ is
convex and increasing for $p \geq2$, and $\kappa$ is a~bounded
function
\begin{eqnarray*}
\mathbb{E} \Bigl\{ \Bigl[\sup_{g \in\mathcal{C}}
|S^N_{\varepsilon}(F) (\log g ) | \Bigr]^{p/2}
\Bigr\}^{2/p} &\leq&\frac{2}{a}\mathbb{E} \Bigl\{ \Bigl[
\sup_{g\in\mathcal{C}}|S^N_{\varepsilon}(F)(g)| \Bigr]^{p/2}
\Bigr\}^{2/p}\\
&&{}+ \frac{2}{a}\mathbb{E} \{ |S^N_{\varepsilon}(F)(1)|^{p/2}
\}^{2/p}.
\end{eqnarray*}
Applying the same technique used to prove Lemma~\ref{lem:2} we have
\[
\mathbb{E} \{ |S^N_{\varepsilon}(F)(1)|^{p/2} \}^{2/p} =
\mathbb{E} \Biggl\{ \Biggl|\frac{1}{N}\sum_{i=1}^N
\varepsilon_{i} \Biggr|^{p/2} \Biggr\}^{2/p} \leq\frac{2
c^{2/p}(p/2)}{\sqrt{N}}.
\]
On the other hand, using the representation of $g \in\mathcal{C}$
and exchanging the order of integration and summation
\begin{eqnarray*}
|S^N_{\varepsilon}(F)(g)| &=& \Biggl| \frac{1}{N} \sum_{i=1}^N
\varepsilon_i \int_{\theta\in\Theta} \phi_{\theta}(\xi_i) \bbP
(\rd\theta) \Biggr| \\
&\leq&\sup_{\theta\in\Theta} \Biggl| \frac{1}{N}
\sum_{i=1}^N
\varepsilon_i\phi_{\theta}(\xi_i) \Biggr|,
\end{eqnarray*}
and we conclude
\[
\mathbb{E} \Bigl\{
\Bigl[\sup_{g\in\mathcal{C}}|S^N_{\varepsilon}(F)(g)| \Bigr]^{p/2}
\Bigr\}^{2/p} \leq\mathbb{E} \Bigl\{
\Bigl[\sup_{g\in\mathcal{H}}|S^N_{\varepsilon}(F)(g)| \Bigr]^{p/2}
\Bigr\}^{2/p}.
\]

The Orlicz norm \cite{Vaa96,DelMor04} $\pi_{\psi_{p}}(Y)$ of a
random variable $Y$ is defined, for a~nondecreasing convex function
$\psi_{p}(x) = e^{x^p} - 1$, as
\[
\pi_{\psi_{p}}(Y) = \inf\bigl\{C > 0 \dvtx \bbE\{ \psi_{p}(|Y| / C) \} \leq
1 \bigr\}.
\]
By Hoeffding's inequality the Rademacher process
$S^N_{\varepsilon}(F)(g)$ is sub-Gaussian for the semimetric
$d_N$ \cite{Vaa96}. Using the fact that $\bbE\{ X^p \}^{1/p} \leq
(p/2)!\pi_{\psi_{2}}(X)$ (see, e.g., Lemma~7.3.5 in \cite{DelMor04}
or \cite{Vaa96}, page~105, Problem~4), we deduce
\[
\bbE\bbE_{\varepsilon}\Bigl \{ \Bigl[\sup_{g\in\mathcal
{H}}|S^N_{\varepsilon}(F)(g)| \Bigr]^{p/2}
\Bigr\}^{2/p} \leq(p/4)!
\bbE\pi_{\psi_{2}}\Bigl(\sup_{g\in\mathcal{H}}|S^N_{\varepsilon}(F)(g)|\Bigr).
\]
In addition, since $S^N_{\varepsilon}(F)(g)$ is sub-Gaussian, we
have for some universal constant $C$ (see proof of Corollary~2.2.8
in \cite{Vaa96})
\[
\bbE\pi_{\psi_{2}}\Bigl(\sup_{g\in\mathcal{H}}|S^N_{\varepsilon}(F)(g)|\Bigr)
\leq\frac{C}{\sqrt{N}}\bbE\int_{0}^b \sqrt{\log\bigl(1 +
\mathcal{D}(\varepsilon, \mathcal{H}, d_N) \bigr)} \,\rd\varepsilon.
\]
Combining the above we have
\begin{eqnarray*}
&&\mathbb{E} \{ |[\widehat{\mathcal{G}}^{N_{\rmp}} - F](h)
|^{p} \}^{1/p} \\&& \qquad \leq\sqrt{2}\biggl [\frac{8}{a\sqrt{N}} \biggl(2
c^{2/p}(p/2)   + (p/4)!C \bbE\int_{0}^b \sqrt{\log
\bigl(1 +
\mathcal{D}(\varepsilon, \mathcal{H}, d_N) \bigr)} \,\rd\varepsilon
\biggr)\\
&&\hspace*{244pt}{} + \frac{\gamma c^2_{f,\mathcal{C}}}{{N_{\rmp}}} + D(f
\| \mathcal{C}) \biggr]^{1/2}.
\end{eqnarray*}
Finally, suppose $1 \leq p < 2$. In this case using Jensen's
inequality we have
\[
\mathbb{E} \{ D(f \| \widehat g^{N_{\rmp}}) ^{p/2}
\}^{2/p} \leq\mathbb{E} \{ D(f \| \widehat
g^{N_{\rmp}}) \}.
\]
Thus the above analysis applies if we choose $p=2$, and the proof is
now complete.
\end{pf*}

\begin{pf*}{Proof of Theorem~\ref{th:Lp_param_time_uni}}
Using the same argument as in Theorem~\ref{th:Lp_time_uni} we have
\begin{eqnarray*}
&&\mathbb{E} \{ |[\widehat\eta_t^{\prime} - \eta_t^{\prime
}](h_t)|^p \}^{1/p} \\
&& \qquad\leq\frac{2 - \epsilon_{\rmu}(M) \epsilon_{\rmu}^{m
K_{\rmu}}(G)}{\epsilon_{\rmu}(M) \epsilon_{\rmu}^{m
K_{\rmu}}(G)} \sum_{i=0}^t \beta_{i,t}(P)
\mathbb{E} \{ | [ \widehat\eta_i^{\prime} -
\Phi_{i}^{\prime}(\widehat\eta_{i-1}^{\prime}) ](h_i)
|^{p} \}^{1/p}.
\end{eqnarray*}
Based on~\eqref{eqn:distr_PF_PA},
$\widehat{\mathcal{G}}_{i+1}^{N_{\rmp}} =
\mathbb{W}_{N_{\rmp}} \circ S^N
(\Phi_{i+1}^{\prime}(\widehat\eta_{i}^{\ell_{0,i}^{\prime}}))$ and
Minkowski inequality we have the decomposition for each individual
expectation under the sum above:
\begin{eqnarray*}
&&\mathbb{E} \{ | [\widehat\eta_i^{\prime} -
\Phi_{i}^{\prime}(\widehat\eta_{i-1}^{\prime}) ](h_i)
|^p
\}^{ {1}/{p}} \\
&& \qquad\leq\mathbb{E} \{ | \Delta^{\prime}_{i} [
S^{N}(\widehat{\mathcal{G}_i}^{N_{\rmp}}) -
\widehat{\mathcal{G}_i}^{N_{\rmp}} ](h_i)
\\
&&\hphantom{\mathbb{E} \{} \qquad \quad{}+ (1-\Delta^{\prime}_{i})
[S^{N}(\Phi_{i}^{\prime}(\widehat\eta_{i-1}^{\prime})) -
\Phi_{i}^{\prime}(\widehat\eta_{i-1}^{\prime})
](h_i) |^{p} \}^{ {1}/{p}}\\
&& \qquad \quad{}+ \mathbb{E} \{ | \Delta^{\prime}_{i} [
\widehat{\mathcal{G}_i}^{N_{\rmp}} -
\Phi_{i}^{\prime}(\widehat\eta_{i-1}^{\prime}) ](h_i)
|^{p} \}^{{1}/{p}}.
\end{eqnarray*}
Using the same conditioning argument as in
Theorem~\ref{th:Lp_time_uni} and applying Corollary~\ref{cor:DLp}
based on the assumption $(\mathcal{H})_{\rmu}$ to the last
term we have
\begin{eqnarray*}
&&\mathbb{E} \{ | \Delta^{\prime}_{i} [
\widehat{\mathcal{G}_i}^{N_{\rmp}} -
\Phi_{i}^{\prime}(\widehat\eta_{i-1}^{\prime})
](h_i) |^{p} \}^{1/p} \\
&& \qquad= \mathbb{E}\bigl \{ \mathbb{E} \{ \Delta^{\prime}_{i} |
[ \widehat{\mathcal{G}_i}^{N_{\rmp}} -
\Phi_{i}^{\prime}(\widehat\eta_{i-1}^{\prime}) ](h_i)
|^{p} | \mathcal{F}_{i-1} ,
Y_{i-1}^{\mathcal{S}_{\ell_{i-1}}} =
y_{i-1}^{\mathcal{S}_{\ell_{i-1}}}
\} \bigr\}^{1/p} \\
&& \qquad\leq q_{i}^{1/p}\sqrt{2} \biggl[\frac{8}{a_{i}\sqrt{N}} \biggl(2
c^{2/p}(p/2) \\
&&\hphantom{q_{i}^{1/p}\sqrt{2} \biggl[\frac{8}{a_{i}\sqrt{N}} \biggl(} \qquad  \quad
{}+ (p/4)!C \bbE\int_{0}^{b_{i}} \sqrt{\log(1 +
\mathcal{D}(\varepsilon, \mathcal{H}_i, d_N) )} \,\rd
\varepsilon
\biggr) \\
&&\hspace*{148.5pt}  \qquad\quad{}+ 4\log\bigl(3\sqrt{e}(b_{i}/a_{i})\bigr)
\frac{(b_{i}/a_{i})^2}{{N_{\rmp}}} \biggr]^{1/2}.
\end{eqnarray*}
Next we apply Lemma~\ref{lem:2} and the same conditioning argument
as in Theorem~\ref{th:Exp_time_uni} to the remaining term and
conclude that since $q_{i} \leq q_{\rmu}$, then for any $i \geq
0$ we have the time-uniform estimate
\begin{eqnarray*}
&&\mathbb{E} \{ | [\widehat\eta_i^{\prime} -
\Phi_{i}^{\prime}(\widehat\eta_{i-1}^{\prime}) ](h_i)
|^p \}^{1/p}\\
&& \qquad \leq\frac{c^{1/p}(p)}{\sqrt{N}}\\
&& \qquad  \quad {}  + q_{\rmu}^{1/p}
\sqrt{2} \\
&& \qquad  \quad\hphantom{{}  +} {}\times\biggl[\frac{8}{a_{\rmu}\sqrt{N}} \biggl(2
c^{2/p}(p/2)\\
&& \hphantom{{}  + {}\times\biggl[\frac{8}{a_{\rmu}\sqrt{N}} \biggl(}
 \qquad  \quad {} +
(p/4)!C \sup_{i\ge0}\bbE\int_{0}^{b_{i}} \sqrt{\log
\bigl(1 +
\mathcal{D}(\varepsilon, \mathcal{H}_i, d_N) \bigr)} \,\rd
\varepsilon
\biggr) \\
&&\hspace*{187pt}{}+ 4\log\bigl(3\sqrt{e}(b_{\rmu}/a_{\rmu})\bigr)
\frac{(b_{\rmu}/a_{\rmu})^2}{{N_{\rmp}}}
\biggr]^{1/2}.
\end{eqnarray*}
This along with~\eqref{eqn:dobrushin_onestep_bound}
and~\eqref{eqn:const_sum_bound} completes the proof of theorem.
\end{pf*}


%

\printaddresses

\end{document}